\documentclass[a4paper,11pt,twoside]{amsart}
\usepackage[english]{babel}
\usepackage[utf8]{inputenc}
\usepackage[a4paper,inner=3cm,outer=3cm,top=4cm,bottom=4cm,pdftex]{geometry}
\usepackage{fancyhdr}
\usepackage{fancyhdr}
\usepackage{titlesec}
\usepackage{enumerate}
\usepackage{color}
\usepackage{bold-extra}
\titleformat{\section}{\normalfont\scshape\centering}{\thesection}{1em}{}
\titleformat{\subsection}{\bfseries}{\thesubsection}{1em}{}

\usepackage{color}
\usepackage{bold-extra}
\usepackage{ mathrsfs }

\usepackage{comment}
\usepackage{graphics}
\usepackage{aliascnt}
\usepackage[pdftex,citecolor=green,linkcolor=red]{hyperref}

\usepackage{amsmath}
\usepackage{amsfonts}
\usepackage{amssymb}
\usepackage{amsthm}
\usepackage{comment}
\usepackage{mathtools}

\newtheorem{theorem}{Theorem}[section]

\newtheorem{lemma}[theorem]{Lemma}
\newtheorem{proposition}[theorem]{Proposition}
\theoremstyle{definition}

\newtheorem{remark}[theorem]{Remark}

\numberwithin{equation}{section}

\renewcommand{\Re}{\textnormal{Re}}

\renewcommand\d{\textnormal{d}}

\let\oldpmod\pmod
\renewcommand{\pmod}[1]{\hspace{-0.1cm}\oldpmod {#1}}

\begin{document}

\title{Products of primes in arithmetic progressions}

\author{Kaisa Matom\"aki}
\address{Department of Mathematics and Statistics, University of Turku, 20014 Turku, Finland}
\email{ksmato@utu.fi}

\author{Joni Ter\"{a}v\"{a}inen}
\address{Department of Mathematics and Statistics, University of Turku, 20014 Turku, Finland}
\email{joni.p.teravainen@gmail.com}

\begin{abstract} A conjecture of Erd\H{o}s states that, for any large prime $q$, every reduced residue class $\pmod q$ can be represented as a product $p_1p_2$ of two primes $p_1,p_2\leq q$. We establish a ternary version of this conjecture, showing that, for any sufficiently large cube-free integer $q$, every reduced residue class $\pmod q$ can be written as $p_1p_2p_3$ with $p_1,p_2,p_3\leq q$ primes. We also show that, for any $\varepsilon > 0$ and any sufficiently large integer $q$, at least $(2/3-\varepsilon)\varphi(q)$ reduced residue classes $\pmod q$ can be represented as a product $p_1 p_2$ of two primes $p_1, p_2 \leq q$.

The problems naturally reduce to studying character sums. The main innovation in the paper is the establishment of a multiplicative dense model theorem for character sums over primes in the spirit of the transference principle. In order to deal with possible local obstructions we establish bounds for the logarithmic density of primes in certain unions of cosets of subgroups of $\mathbb{Z}_q^\times$ of small index and study in detail the exceptional case that there exists a quadratic character $\psi \pmod{q}$ such that $\psi(p) = -1$ for very many primes $p \leq q$. 
\end{abstract}

\maketitle

\section{Introduction}
For integers $k,q\geq 1$ and a real number $N \geq 1$, write 
\[
E_k(N) = \{a \in \mathbb{Z}_q^{\times} \colon a \equiv p_1 \dotsm p_k \pmod{q} \text{ for some primes $p_1, \dotsc, p_k \leq N$}\}, 
\]
where $\mathbb{Z}_q^{\times}$ is the set of reduced residue classes $\pmod q$. We study the size of $E_k(N)$ for $k\in \{2,3\}$. 

Let us first note that when $k=1$, Linnik's problem can be formulated as asking for the smallest $L$ such that $E_1(Cq^{L})=\mathbb{Z}_q^{\times}$ for some constant $C$ and all large enough $q$. Linnik~\cite{Linnik1, Linnik2} first proved the existence of such constant $L$ and the best known exponent today is $L=5$ by a result of Xylouris~\cite{xylouris} (which is a slight improvement of an earlier result of Heath-Brown~\cite{Heath-Brown2}).  

Concerning the case $k=2$, Erd\H{o}s conjectured (see~\cite{Erdos}) that $E_2(q)=\mathbb{Z}_q^{\times}$ for all large enough primes $q$. This remains open even under the generalized Riemann hypothesis. In this paper we establish a ternary variant of Erd\H{o}s' conjecture:
\begin{theorem}
\label{th:E3q}
\begin{enumerate}[(i)]
\item Let $q \in \mathbb{N}$ be cube-free and sufficiently large. Then
\[
E_3(q) = \mathbb{Z}_q^{\times}.
\]
\item Let $\varepsilon>0$ and let $q \in \mathbb{N}$ be sufficiently large in terms of $\varepsilon$. Then
\[
E_3(q^{1+\varepsilon}) = \mathbb{Z}_q^{\times}.
\]
\end{enumerate}
\end{theorem}

Previously, Walker had shown in~\cite{walker} that $E_{48}(q)=\mathbb{Z}_q^{\times}$ and later in~\cite{Walker-PhD} that $E_{20}(q)=\mathbb{Z}_q^{\times}$. This was improved independently by Szab\'o~\cite{szabo} and Zhao~\cite{Zhao} to $E_6(q) = \mathbb{Z}_q^\times$ (Szab\'o~\cite{szabo} needed to assume that $q$ is prime), so we get significantly closer to the conjecture of Erd\H{o}s.

On the other hand, several authors have considered the least $N$ for which  $E_3(N)=\mathbb{Z}_q^{\times}$. The best results in the literature are that $N = 650q^3$ works for all $q \geq 2$ due to Ramar\'e, Srivastav and Serra~\cite{RSS}, and that $N = q^{6/5+\varepsilon}$ works for all sufficiently large $q$ due to Szab\'o~\cite{szabo} (improving on~\cite{BRS}). Theorem~\ref{th:E3q} shows that $N = q$ works for every sufficiently large cube-free $q$, and that $N=q^{1+\varepsilon}$ works for every $q$ that is sufficiently large in terms of $\varepsilon$. Actually our techniques would allow one to establish that, for some quite small $\alpha > 0$, one has $E_3(q^{1-\alpha}) = \mathbb{Z}_q^\times$ for all sufficiently large cube-free $q$.

We also note that Klurman, Mangerel and the second author~\cite{KMT} proved that, once $\varepsilon > 0$ is sufficiently small, we have $E_3(q)=\mathbb{Z}_q^{\times}$ for all large enough $q^{\varepsilon}$-smooth numbers $q$, and also $E_3(q)=\mathbb{Z}_q^{\times}$ for all but $O(1)$ primes $q\in [Q^{1/2},Q]$ for any $Q\geq 1$. The method in~\cite{KMT} depends on the quality of the zero-free regions available for Dirichlet $L$-functions $\pmod q$, and hence it does not work for arbitrary $q$. In case when $q$ is smooth, one has better character sum estimates, and due to this our methods would allow one to prove $E_3(q^{1-\alpha})=\mathbb{Z}_q^{\times}$ for somewhat larger $\alpha$ than in the cube-free case.

\begin{remark}
In Theorem~\ref{th:E3q}(i), the assumption that $q \in \mathbb{N}$ is cube-free and sufficiently large could be weakened for instance to the assumption that, for some $\varepsilon>0$, the cube-full part of $q$ is $\leq q^{1/9-\varepsilon}$ and $q$ is sufficiently large in terms of $\varepsilon$. Let us briefly sketch why this is the case; we leave the details to the interested reader. In the proof of Theorem~\ref{th:E3q}(ii), we look at $E_3(q^{1+100\varepsilon})$ instead of $E_3(q)$ just because we need Proposition~\ref{prop:A'binter} to be applicable with $\theta=2/3+10\varepsilon$. But, as noted in Remark~\ref{rmk:burgess}, the only property of $\theta$ needed for Proposition~\ref{prop:A'binter} is that 
\begin{align}\label{eq:burgessbound}
\max_{\substack{\chi\pmod q\\\chi\neq \chi_0}}\, \max_{N^{1-\theta} \leq y \leq N} \frac{1}{y}\left|\sum_{n\leq y}\chi(n)\right|\ll N^{-\varepsilon_0}    
\end{align}
for some $\varepsilon_0>0$. By a direct consequence of the Burgess bound (see~\cite[(12.56)]{iw-kow}), we do have~\eqref{eq:burgessbound} with $N = q$ and $\theta = 2/3+10\varepsilon$ under the assumption that the cube-full part of $q$ is $\leq q^{1/9-100\varepsilon}$. Now the claim follows by adjusting $\varepsilon$. 
\end{remark}

We also consider another approximation toward the conjecture of Erd\H{o}s, namely the problem of lower-bounding the density of $E_2(q)$ inside $\mathbb{Z}_q^{\times}$, and prove the following.

\begin{theorem}\label{thm:E2>1/2} Let $\varepsilon>0$ and let $q \in \mathbb{N}$ be sufficiently large in terms of $\varepsilon$.
\begin{enumerate}[(i)] 
\item We have
\[
|E_2(q)| \geq \left(\frac{2}{3}-\varepsilon\right) \varphi(q).
\]
\item Assume that $q$ is cube-free. Then
\[
|E_2(q)| \geq \left(\frac{11}{16}-\varepsilon\right) \varphi(q).
\] 
\end{enumerate}
\end{theorem} 
Here $2/3 \approx 0.666\ldots$ and $11/16=0.6875$ whereas the previous record was due to Szab\'o~\cite{szabo} with $3/8 = 0.375$ in place of $11/16$ in the case that $q$ is cube-free (for general $q$, Szab\'o's method gives proportion $1/3-\varepsilon$). Note that Theorem~\ref{thm:E2>1/2}(i) immediately implies (taking $\varepsilon<1/6$) the previously unknown result that $E_4(q)=\mathbb{Z}_q^{\times}$ for all large $q$. 

In addition to studying $E_k(N)$ as in Theorems~\ref{th:E3q} and~\ref{thm:E2>1/2}, there is still another approximation toward the conjecture of Erd\H{o}s --- one can ask what is the least integer $k$ such that every $a \in \mathbb{Z}_q^\times$ can be written as $a = p n$ with $p, n \leq q$ and $n$ having at most $k$ prime factors. Our methods are not applicable to this binary problem that is in the literature approached through sieving the variable $n$. The best result on this problem is due to Zhao~\cite{Zhao} with $k = 5$. Zhao~\cite{Zhao} also obtained that every $a \in \mathbb{Z}_q^\times$ can be written as $a = p_1 p_2 n$ with $p_1, p_2, n \leq q$ and $n$ having at most two prime factors. Of course our Theorem~\ref{th:E3q} improves this last result to $n$ being prime in case $q$ is cube-free.

We shall give outlines of proofs in the following subsection, but let us here very briefly mention the main components of the proofs. The main innovation is the establishment of a multiplicative dense model theorem for character sums over primes in the spirit of the transference principle. The dense model theorem transfers our problems to additive combinatorial problems concerning product sets of positive density subsets of $\mathbb{Z}_q^\times$. To consider these latter problems, we use Kneser's theorem after which we need to deal with the case that the primes are concentrated in a union of cosets of a small index subgroup of $\mathbb{Z}_q^\times$. To deal with this case, we use some sieve theory to obtain lower bounds for the number of primes in unions of cosets of small index. In particular we need to take special care of the case when there exists an exceptional quadratic character $\psi \pmod{q}$ for which $\psi(p) = -1$ for very many primes $p \leq q$.

The multiplicative transference principle established here is likely to have other applications as well. In a forthcoming work joint with Merikoski, we shall consider Linnik's problem using sieve ideas and some of the techniques used in this work.

\subsection{Proof outline}
The starting point of the proof of Theorem~\ref{th:E3q}(i) is (essentially) the orthogonality identity
\begin{align}
\label{eq:starting2}
\sum_{\substack{p_1,p_2,p_3 \leq q \\ p_1 p_2 p_3 \equiv a \pmod{q} }} \log p_1 \log p_2 \log p_3 = \frac{1}{\varphi(q)} \sum_{\chi \pmod{q}} \left(\sum_{p \leq q} \overline{\chi}(p) \log p\right)^3 \chi(a).
\end{align}
Separating the contribution of the principal character $\chi_0\pmod q$ and applying the prime number theorem, this becomes
\begin{align*}
(1+o(1))\frac{q^3}{\varphi(q)}+\frac{1}{\varphi(q)} \sum_{\substack{\chi \pmod{q}\\\chi\neq \chi_0}} \left(\sum_{p \leq q} \overline{\chi}(p) \log p\right)^3 \chi(a).
\end{align*}
If we knew that we have
\begin{equation}
\label{eq:dream}
\max_{\substack{\chi\pmod q\\\chi\neq \chi_0}}\left|\sum_{p \leq q} \overline{\chi}(p) \log p\right| \ll \frac{q}{(\log q)^{10}},
\end{equation}
say, the claim would follow quickly by applying to~\eqref{eq:starting2} this bound together with the standard mean value theorem for character sums (see Lemma~\ref{le:MVT} below). However, unfortunately, there are no chances to unconditionally establish~\eqref{eq:dream} --- note that the sum is only up to $p \leq q$ with $q$ the conductor of the character.

We overcome this obstacle by establishing a dense model theorem for character sums over primes. This is motivated by the transference approach to additive problems, introduced by Green~\cite{green-annals} in the case of counting arithmetic progressions and generalized in~\cite{Li-PanGoldbach},~\cite{shao},~\cite{matomaki-shao},~\cite{mms} to deal with the ternary Goldbach equation. 

In particular Shao~\cite{shao} considered the ternary Goldbach problem for positive density subsets of the primes. More precisely, he proved that if a set $A \subseteq \mathbb{P}$ has lower relative density greater than $5/8$, then each sufficiently large odd integer $N$ can be represented as a sum of three primes from $A$. Taking $W=\prod_{p\leq w}p$ with $w$ growing to infinity and splitting into arithmetic progressions $\pmod{W}$, this problem reduces to studying a sumset problem in $\mathbb{Z}_W$ and to studying, for $b_i$ coprime to $W$,
\begin{align}\begin{split}
\label{eq:circlemethod}
&\sum_{\substack{n_1,n_2,n_3\leq N/W\\(Wn_1+b_1)+(Wn_2+b_2)+(Wn_3+b_3)=N \\ Wn_1+b_1, Wn_2+b_2, Wn_3+b_3 \in A}} \log n_1 \log n_2 \log n_3\\
&= \int_0^1 \left(\prod_{i=1}^3 \sum_{\substack{n \leq N/W}} \frac{W}{\varphi(W)}f_{b_i}(n) e(\alpha (Wn+b_i))\right) e(-\alpha N) \,\textnormal{d}\alpha,
\end{split}
\end{align}
where $f_{b_i} \colon [1, N/W] \cap \mathbb{N} \to \mathbb{R}_{\geq 0}$ are defined by $f_{b_i}(n) = \frac{\varphi(W)}{W}(\log n) \mathbf{1}_{Wn+b \in A}$. 

Now, for each $b$ with $(b, W) = 1$, one can find a dense model $g_b \colon [1, N/W] \cap \mathbb{N} \to [0,1]$ such that $\widehat{g_b}$ is close to $\widehat{f_b}$ in sup norm. Then one can replace $f_{b_i}$ by $g_{b_i}$ in~\eqref{eq:circlemethod} (since one has also an $L^r$ bound for $\widehat{f_b}$ and $\widehat{g_b}$ for some $r\in (2,3)$) and reduce the problem to studying sumsets of positive density subsets of the integers. The corresponding sumset problems can in principle be solved unless there are some "local" obstructions, such as $A$ being concentrated in an arithmetic progression or a Bohr set.

We argue similarly. In particular we essentially let $f \colon [1, q] \cap \mathbb{N} \to \mathbb{R}_{\geq 0}$ be defined through $f(n) = 1_{n \in \mathbb{P}} \log n$ and prove a multiplicative analogue of a dense model theorem, showing, using the Burgess bound for cube-free $q$ and a linear sieve majorant, that there exists $g \colon \mathbb{Z}_q^\times \to [0,\frac{8}{3}+\varepsilon]$ such that
\[
\max_{\chi\pmod q}\left|\sum_{n \leq q} f(n) \overline{\chi}(n) - \frac{q}{\varphi(q)}\sum_{a \in \mathbb{Z}_q^\times} g(a) \overline{\chi}(a)\right| = o(\varphi(q)).
\]
Moreover, by a Hal{\'a}sz--Montgomery type inequality and the mean value theorem (again in the form of Lemma~\ref{le:MVT}) we can show that $\widehat{f}$ and $\widehat{g}$ satisfy suitable $L^r$ bounds with $r \in (2, 3)$ (which we will formulate as large value estimates rather than $L^r$ bounds, see Remark~\ref{rem:Lrbound} below). 
Combining these facts with~\eqref{eq:starting2}, we can show that
\begin{align*}
\sum_{\substack{p_1,p_2,p_3 \leq q \\  p_1 p_2 p_3 \equiv a \pmod{q} }} \log p_1 \log p_2 \log p_3 &= \sum_{\substack{n_1,n_2,n_3 \leq q \\ n_1 n_2 n_3 \equiv a \pmod{q}}} f(n_1) f(n_2) f(n_3)  \\
&= \frac{q^3}{\varphi(q)^3} \sum_{\substack{a_1,a_2,a_3 \in \mathbb{Z}_q^\times \\ a_1 a_2 a_3 = a}} g(a_1)g(a_2)g(a_3) + o\left(\frac{q^3}{\varphi(q)}\right).
\end{align*}
Writing 
$$A = \{a \in \mathbb{Z}_q^\times \colon |g(a)| \geq \varepsilon/10\},$$
we obtain that $a \in E_3(q)$ if $(\mathbf{1}_A \ast \mathbf{1}_A \ast \mathbf{1}_A)(a) \gg \varphi(q)^2$. Furthermore, it follows from the construction of $g$ that
\[
|A| \geq \left(\frac{3}{8}-\varepsilon\right)\varphi(q).
\]

Using a popular sums version of Kneser's theorem (Lemma~\ref{le:popularkneser}), we have $(\mathbf{1}_A \ast \mathbf{1}_A \ast \mathbf{1}_A)(a) \gg \varphi(q)^2$ for all $a\in \mathbb{Z}_q^{\times}$ (and hence $E_3(q) = \mathbb{Z}_q^\times$) unless $A$ is essentially (up to a set of cardinality $o(\varphi(q))$) stuck in $k+1$ cosets of a subgroup $H\leq \mathbb{Z}_q^{\times}$ of index $3k+2$, for some $k \in \{0, 1, 2\}$, and $A \cdot A$ is essentially the union of $2k+1$ cosets of $H$. The dense model $g$ is defined in such a way that this means that also the primes $p \leq q$ are essentially stuck in these cosets. 

We are now left with three problematic scenarios: 
\begin{enumerate}[(i)]
    \item Almost all primes up to $q$ are contained in $3$ cosets of an index $8$ subgroup 
 of $\mathbb{Z}_q^{\times}$.
 \item Almost all primes up to $q$ are contained in $2$ cosets of an index $5$ subgroup 
 of $\mathbb{Z}_q^{\times}$.
 \item Almost all primes up to $q$ are contained in a coset of an index $2$ subgroup 
 of $\mathbb{Z}_q^{\times}$.
\end{enumerate}

None of these scenarios should happen in reality, but the existing bounds for least primes in cosets (for example~\cite[Theorem 2.4]{BRS}) do not seem to rule them out, and scenario (iii) could in fact happen if there were Siegel zeros.

Using a lemma of Heath-Brown on character sums over primes, we will be able to show that, for at least one of the remaining $2k+1$ cosets not occupied by $A$, say $a_0 H$, one has
\begin{equation}
\label{eq:sketchplow}
\sum_{\substack{p \in (q^{\alpha-\varepsilon}, q^\alpha] \\ p \in a_0 H}} \frac{1}{p} \gg 1 \quad \text{for some $\alpha \in [2\varepsilon, 1]$},
\end{equation}
except in the case when $k=0$ and $A$ is essentially contained in the coset $a_1 H \neq H$. This last case, corresponding to (iii) above, requires some special care as we will explain below. Also the bound~\eqref{eq:sketchplow} is not quite sufficient to refute the scenarios (i) or (ii), since above we would need a positive natural density of primes $p \leq q$ in the coset $a_0H$. However, in case~\eqref{eq:sketchplow} holds we study the logarithmically weighted count
\[
\sum_{\substack{p_1 p_2 p_3\equiv a\pmod q \\ p_2, p_3 \leq q \\ q^{\alpha-\varepsilon} < p_1 \leq q^{\alpha}}} \frac{\log p_2 \log p_3}{p_1} = \frac{1}{\varphi(q)} \sum_{\chi \pmod{q}} \left(\sum_{p \leq q} \overline{\chi}(p) \log p\right)^2 \left(\sum_{\substack{q^{\alpha-\varepsilon} < p_1 \leq q^{\alpha}}} \frac{\overline{\chi}(p_1)}{p_1}\right) \chi(a).
\]
In this case we use the dense model only for the primes $p_2, p_3$. Using again large value results for character sums (also for an appropriate moment of the character sum of $p_1$), we can show that
\begin{align*}
\sum_{\substack{p_1 p_2 p_3\equiv a\pmod q \\ p_2, p_3 \leq q \\ q^{\alpha-\varepsilon} < p_1 \leq q^{\alpha}}} \frac{\log p_2 \log p_3}{p_1} &= \left(\frac{q}{\varphi(q)}\right)^2\sum_{\substack{p_1 a_2 a_3=a \\ q^{\alpha-\varepsilon} < p_1 \leq q^{\alpha} \\ a_2, a_3 \in \mathbb{Z}_q^\times}} \frac{g(a_2) g(a_3)}{p_1} + o\left(\frac{q^2}{\varphi(q)}\right) \\
&\geq \varepsilon^2 \left(\frac{q}{\varphi(q)}\right)^2\sum_{\substack{p_1 a_2 a_3=a \\ q^{\alpha-\varepsilon} < p_1 \leq q^{\alpha} \\ a_2, a_3 \in A}} \frac{1}{p_1} + o\left(\frac{q^2}{\varphi(q)}\right),
\end{align*}
with $A$ as above. Since $A \cdot A$ is now essentially a union of $2k+1$ cosets, this together with~\eqref{eq:sketchplow} will allow us to finish the proof.

On the other hand, in case that $k=0$ and $A$ is essentially contained in the coset $a_1 H \neq H$ (so that we are in scenario (iii)), we can use sieve-theoretic arguments to show (see Lemma~\ref{le:QR}) that either~\eqref{eq:sketchplow} holds with $a_0 \in H$ or there exists $y\in [q^{5/7},q]$ such that
\begin{align}
\label{eq:sketchplowL1chi}
\sum_{\substack{p\leq y\\ \psi(p) = 1}} 1 \gg y L(1, \psi) \frac{\varphi(q)}{q}\prod_{\substack{2<p < q \\ \psi(p) = 1}} \left(1-\frac{2}{p}\right),
\end{align}
where $\psi$ is the quadratic character $\pmod{q}$ for which $\psi(h) = 1$ for every $h \in H$. The lower bound in~\eqref{eq:sketchplowL1chi} can be a lot smaller than $y/\log y$ if $L(1,\psi)$ is small (and hence we cannot rule out scenario (iii)), but we can use a sieve to construct a pseudorandom majorant for the primes $p\leq q$ with $\psi(p)=1$ that takes into account this smaller density of primes with $\psi(p)=1$, and this allows us to prove a lossless Hal\'asz--Montgomery type large value theorem for primes $p\leq y$ with $\psi(p) = 1$ (see Lemma~\ref{le:Hal-MonL1chi}). This in turn will allow us to make an appropriate transference and finish the proof.

In case $q$ is not cube-free, taking into account the version of Burgess for arbitrary moduli the previous argument would only give us
\begin{equation}
\label{eq:sketchAlow}
|A| \geq \left(\frac{1}{3}-\varepsilon\right) \varphi(q).
\end{equation}
This is not sufficient for a successful application of Kneser's theorem since we might have that $|A\cdot A\cdot A| \leq (1-3\varepsilon) \varphi(q)$ even if $A$ is not stuck in a union of arithmetic progressions. But once we study $E_3(q^{1+100\varepsilon})$ we can replace $1/3-\varepsilon$ by $1/3+\varepsilon$ in~\eqref{eq:sketchAlow} and the previous method works, since we can also show that $A$ occupies at least proportion $3/8-\varepsilon$ of cosets of any subgroup of $\mathbb{Z}_q^\times$ with fixed index.

In the proof of Theorem~\ref{thm:E2>1/2}, we argue similarly to the first part of the proof of Theorem~\ref{th:E3q} to obtain that either the claim holds or we have a similar local obstruction as in the ternary case. In this case, we are able to prove lower bounds for the number of products of two primes in the cosets not occupied by $A \cdot A$ through studying multiplicative energy.

\section{Notation}

\subsection{General notation}

Throughout the paper, we assume that $\varepsilon > 0$ is sufficiently small. We write $x\sim y$ as a shorthand for $y< x\leq 2y$. By $\varphi, \Lambda,$ and $\tau_k$ we denote the Euler totient function, the von Mangoldt function, and the $k$-fold divisor function (with $\tau = \tau_2$), as usual. For $z\geq 1$, we write $P(z)=\prod_{p<z}p$. We write $(a,b)$ for the greatest common divisor of integers $a$ and $b$. We write
\begin{equation}
\label{eq:Ndef}
[N]_q = \{n \leq N \colon (n, q) = 1\}.
\end{equation}
By~\eqref{eq:Simplegcd} below we have that
\begin{equation}
\label{eq:Nqsize}
|[N]_q| = \frac{\varphi(q)}{q} N + O(\tau(q)).
\end{equation}
For a finite, nonempty set $A$ and a function $f:A\to \mathbb{C}$, we use the notation
\[
\mathbb{E}_{n \in A}f(n) = \frac{1}{|A|}\sum_{n\in A}f(n).
\]

\subsection{Group theory} 

For a finite abelian group $G$ written in multiplicative notation and subsets $A,B\subseteq G$, define the product set $A\cdot B=\{ab:\,\, a\in A, b\in B\}$, and for $c\in G$ define the dilation $cA = c\cdot A=\{ca:\,\, a\in G\}$. For a subset $A \subseteq G$, the stabilizer of $A$ is the subgroup $\{g \in G \colon g \cdot A = A\}$.

The convolution of two functions $f,g:G\to \mathbb{C}$ is defined as
\begin{align*}
(f*g)(a)=\sum_{\substack{b,c\in G\\bc=a}}f(b)g(c). 
\end{align*}

By $\chi_0\pmod q$ we denote the principal Dirichlet character $\pmod q$. For a subgroup $H\leq \mathbb{Z}_q^\times$ and an element $b \in \mathbb{Z}_q^\times$, we abuse the notation by writing, for $n \in \mathbb{Z}$, $n \in bH$ if $n \pmod{q} \in bH$.

We will frequently use that the orthogonality of characters implies that, for any $a \in \mathbb{Z}_q^\times$ and $n \in \mathbb{N}$ or $n \in \mathbb{Z}_q^\times$,
\[
\mathbf{1}_{n = a \pmod{q}} = \frac{1}{\varphi(q)} \sum_{\chi \pmod{q}} \overline{\chi}(n) \chi(a)
\]
and, for any $H \leq \mathbb{Z}_q^\times$ and $b \in \mathbb{Z}_q^\times$,
\begin{equation}
\label{eq:incoset}
\mathbf{1}_{n \in bH} = \frac{1}{\varphi(q)} \sum_{\substack{\chi \pmod{q} \\ \chi(h) = 1 \text{ for all $h \in H$}}} \overline{\chi}(n) \chi(b).
\end{equation}
Note that the number of characters in the last sum equals the index of $H$ in $\mathbb{Z}_q^\times$.

\section{Auxiliary results}

\subsection{Sieves, Burgess' bound, and products of primes in cosets}

We shall use the linear sieve, see e.g.~\cite[Section 12]{friedlander}.
\begin{lemma}
\label{le:linear}
Let $z\geq 2$ and let $D = z^s$ with $s\geq 1$. There exist coefficients $\lambda_d^\pm$ such that the following hold.
\begin{enumerate}[(i)]
\item $|\lambda_d^\pm| \leq 1$ for every $d \in \mathbb{N}$ and $\lambda_d^\pm$ are supported on $\{d \leq D:\,\, d \mid P(z)\}$.
\item For every $n \in \mathbb{N}$,
\[
\sum_{d \mid n} \lambda^-_d \leq \mathbf{1}_{(n, P(z)) = 1} \leq \sum_{d \mid n} \lambda^+_d. 
\]
\item If $h \colon \mathbb{N} \to [0, 1)$ is a multiplicative function such that 
\[
\prod_{w_1 \leq p < z_1} (1-h(p))^{-1} \leq \frac{\log z_1}{\log w_1} \left(1+O\left(\frac{1}{\log w_1}\right)\right)
\]
for any $z_1 \geq w_1 \geq 2$, then
\begin{align*}
\sum_{d \mid P(z)} \lambda_d^+ h(d) &\leq (F_0(s)+o(1)) \prod_{p < z} (1-h(p)), \\
\sum_{d \mid P(z)} \lambda_d^- h(d) &\geq (f_0(s)+o(1)) \prod_{p < z} (1-h(p)),
\end{align*}
where $F_0$ and $f_0$ are the upper and lower bound linear sieve functions (see e.g.~\cite[(12.1)--(12.2)]{friedlander}), so that in particular
\begin{align}
\begin{aligned}
\label{eq:F0f0def}
F_0(s)&= 2e^\gamma/s \textnormal{  for  } s \in [1, 3],\\ 
f_0(s)&=2e^{\gamma}\log(s-1)/s \textnormal{  for  } s\in [2,4].   
\end{aligned}
\end{align}
\end{enumerate}
\end{lemma}

Applying the sieve in our set-up often leads us to sums over integers coprime to the modulus $q \in \mathbb{N}$. For these we note that by M\"obius inversion we have, for any $y \geq 1$,
\begin{equation}
\label{eq:Simplegcd}
\sum_{\substack{m \leq y \\ (m, q) = 1}} 1 = \sum_{e \mid q} \mu(e) \sum_{m \leq y/e} 1 = y\sum_{e \mid q} \frac{\mu(e)}{e} + O(\tau(q)) = y\frac{\varphi(q)}{q} + O(\tau(q)).
\end{equation}

For character sums we use the Burgess bound.
\begin{lemma}
\label{le:Burgess}
Let $q \in \mathbb{N}$, let $\chi$ be a primitive character $\pmod{q}$ and let $M, N \geq 1$. Then, for any $\varepsilon > 0$
\[
\left|\sum_{M < n \leq M+N} \chi(n)\right| \ll_{\varepsilon} N^{1-\frac{1}{r}} q^{\frac{r+1}{4r^2} + \varepsilon}
\]
for $r \in \{1, 2, 3\}$. If $q$ is cube-free or $\chi$ has bounded order, then this holds for any $r \in \mathbb{N}$.

In particular, for any $\varepsilon > 0$, there exists $\delta = \delta(\varepsilon) > 0$ such that for any $q \in \mathbb{N}$ and any $N \geq q^{1/3+\varepsilon}$ we have
\[
\left|\sum_{n \leq N} \chi(n)\right| \ll_{\varepsilon} N^{1-\delta}.
\]
When $q$ is cube-free or $\chi$ has bounded order, this holds for $N \geq q^{1/4+\varepsilon}$.
\end{lemma}
\begin{proof}
The case $r = 1$ of the first part follows from the P{\'o}lya--Vinogradov inequality (see e.g.~\cite[Theorem 12.4]{iw-kow}). The case $r \geq 2$ of the first part is the Burgess bound, see e.g.~\cite[Theorem 12.5]{iw-kow} for the general and cube-free case and~\cite[Lemma 2.4]{Heath-Brown2} for the bounded order case.

The second part follows from the first by taking $r=3$ in the case of arbitrary $q$ and $r$ large in case $q$ is cube-free or $\chi$ has bounded order.
\end{proof}

Let us prove the following quick consequence of the linear sieve and the Burgess bound.
\begin{lemma}[The number of products of at most two primes in cosets]
\label{le:P2}
Let $\varepsilon>0$ be fixed. Let $q \in \mathbb{N}$, $N \geq q^{3/4+\varepsilon}$, and let $\theta_0 = 1-\varepsilon-\frac{\log q}{4 \log N}$. Let $H \leq \mathbb{Z}_q^\times$ be a subgroup of fixed index $Y$ and let $b \in \mathbb{Z}_q^\times$. Then
\[
(1+o(1))\frac{2 \log(3\theta_0 -1)}{\theta_0} \cdot \frac{N}{Y \log N} \leq \sum_{\substack{n \leq N \\ n \in bH \\ p \mid n \implies p \geq N^{1/3}}} 1 \leq (1+o(1))\frac{2}{\theta_0} \cdot \frac{N}{Y \log N}
\]
\end{lemma}

\begin{proof}
Let us first prove the lower bound. Let $\lambda_d^-$ be as in Lemma~\ref{le:linear} with $D = N^{\theta_0}$ and $s=3\theta_0$ (so that $z = N^{1/3}$). By Lemma~\ref{le:linear}(ii), 
\begin{align}
\label{eq:P2low}
\sum_{\substack{n \leq N \\ n \in bH \\ p \mid n \implies p \geq N^{1/3}}} 1 \geq \sum_{d \leq D} \lambda_d^- \sum_{\substack{m \leq N/d \\ dm \in bH}} 1.
\end{align}
Now we detect the condition $dm \in b\cdot H$ using~\eqref{eq:incoset}. Let $\chi_0, \chi_1, \dotsc, \chi_{Y-1}$ be the characters $\pmod{q}$ for which $\chi_j(h) = 1$ for every $h \in H$. Then
\begin{align*}
\sum_{\substack{m \leq N/d \\ dm \in bH}} 1 &= \sum_{m \leq N/d} \frac{1}{Y} \sum_{j = 0}^{Y-1} \overline{\chi_j}(dm) \chi_j(b) =  \frac{\mathbf{1}_{(d, q) = 1}}{Y} \sum_{\substack{m \leq N/d \\(m, q) = 1}} 1 + \frac{1}{Y} \sum_{j = 1}^{Y-1} \overline{\chi_j}(d) \chi_j(b)  \sum_{\substack{m \leq N/d}} \overline{\chi_j}(m).
\end{align*}

By~\eqref{eq:Simplegcd} and the Burgess bound (Lemma~\ref{le:Burgess}, where we note that each $\chi_j$ has bounded order $\leq Y$) we obtain, for some $\delta=\delta(\varepsilon) > 0$,
\begin{equation}
\label{eq:lodBurg}
\sum_{\substack{m \leq N/d \\ dm \in bH}} 1 = \mathbf{1}_{(d, q) = 1} \frac{1}{Y} \frac{N}{d} \frac{\varphi(q)}{q} + O(\tau(q)) + O((N/d)^{1-\delta}).
\end{equation}
Plugging this into~\eqref{eq:P2low}, we obtain
\[
\sum_{\substack{n \leq N \\ n \in bH \\ p \mid n \implies p > N^{1/3}}} 1 \geq \frac{N}{Y} \cdot \frac{\varphi(q)}{q} \sum_{(d, q) = 1} \frac{\lambda_d^-}{d}+O(D\tau(q)+N^{1-\delta}D^{\delta}).
\]
By Lemma~\ref{le:linear}(iii) and Mertens' theorem, we have
\begin{align*}
\sum_{(d, q) = 1} \frac{\lambda_d^-}{d}\geq  (f_0(s)+o(1))\prod_{p < N^{1/3}}\left(1-\frac{\mathbf{1}_{p\nmid q}}{p}\right)=(e^{-\gamma}f_0(s)+o(1))\frac{q}{\varphi(q)\log N^{1/3}}    
\end{align*}
The lower bound part of the claim now follows from~\eqref{eq:F0f0def}, since $s=3\theta_0\leq 4$ and either $3\theta_0\geq 2$ or the claim is trivial. The upper bound follows completely similarly but using the upper bound linear sieve instead.
\end{proof}

\subsection{Lower bounds on product sets}

Next we collect some results concerning convolutions of characteristic functions and product sets (or equivalently sumsets in additive notation). The following simple lemma gives a lower bound for convolutions on a product set.

\begin{lemma}\label{le:convolution} 
Let $G$ be a finite abelian group.
\begin{enumerate}[(i)]
\item Let $A,B\subseteq G$ be nonempty subsets of $G$. Then we have
\begin{align*}
\mathbf{1}_{A}*\mathbf{1}_{B}(c)\geq |A|+|B|-|G|  \end{align*}
for all $c\in G$.
\item Let $H \leq G$, let $a, b \in G$, and let $A \subseteq aH$ and $B \subseteq bH$. Then, for any $c \in abH$, we have
\begin{align*}
\mathbf{1}_{A}*\mathbf{1}_{B}(c)\geq |A|+|B|-|H|.
\end{align*}
\end{enumerate}
\end{lemma}

\begin{proof}
Part (i) follows since
\begin{align*}
\mathbf{1}_{A}*\mathbf{1}_{B}(c)&=|c A^{-1} \cap B|=|G|-|(cA^{-1})^{c}\cup B^{c}|\\
&\geq |G|-(|G|-|A|)-(|G|-|B|)=|A|+|B|-|G|.
\end{align*}
Part (ii) follows from applying part (i) to the group $H$ and sets $a^{-1} A, b^{-1} B \subseteq H$.
\end{proof}

Kneser's theorem (see e.g.~\cite[Theorem 5.5]{tao-vu}) is a standard tool for studying product sets inside abelian groups.

\begin{lemma}[Kneser's theorem]
\label{le:Kneser}
Let $G$ be a finite abelian group and let $A, B \subseteq G$. Let $H$ be the stabilizer of $A \cdot B$. Then
\[
|A \cdot B| \geq |A \cdot H| + |B \cdot H| - |H|\geq |A|+|B|-|H|.
\]
\end{lemma}

\subsection{Mean values of character sums}
The following lemma is the basic mean value theorem (or large sieve) for characters sums.
\begin{lemma}[Mean value theorem]
\label{le:MVT}
Let $q \in \mathbb{N}$ and $N \geq 2$. Then, for any complex coefficients $a_n$,
\[
\sum_{\chi \pmod{q}} \left|\sum_{n \leq N} a_n \chi(n)\right|^2 \leq (N+\varphi(q)) \sum_{\substack{n \leq N \\ (n, q) = 1}} |a_n|^2.
\]
\end{lemma}
\begin{proof}
Squaring out, using the orthogonality of characters and the inequality $|a_m a_n| \leq (|a_m|^2+|a_n|^2)/2$, we obtain 
\begin{align*}
\sum_{\chi \pmod{q}} \left|\sum_{n \leq N} a_n \chi(n)\right|^2 &\leq \varphi(q) \sum_{\substack{m, n \leq N \\ m \equiv n \pmod{q} \\ (mn, q) = 1}} |a_m a_n| \leq \varphi(q) \sum_{\substack{n \leq N \\ (n, q)=1}} |a_n|^2 \sum_{\substack{m \leq N \\ m \equiv n \pmod{q}}} 1 \\
&\leq \varphi(q) \left(\frac{N}{q}+1\right) \sum_{\substack{n \leq N \\ (n, q) = 1}} |a_n|^2.
\end{align*}
\end{proof}
We shall also use the following large value theorem due to Huxley (see e.g.~\cite[Theorem 9.15 with $s_r(\chi) = 0$ for every $r$]{iw-kow}).
\begin{lemma}
\label{le:Huxley}
Let $q \in \mathbb{N}$ and $N \geq 2$, and let $\chi_1, \dotsc, \chi_R$ be distinct Dirichlet characters of modulus $q$. Then, for any complex coefficients $a_n$,
\[
\sum_{j=1}^R \left|\sum_{n \leq N} a_n \chi_j(n)\right|^2 \ll (N+R^{2/3} N^{1/3} q^{1/3}) (\log (qN))^6 \sum_{n \leq N} |a_n|^2.
\]
\end{lemma}

The following lemma gives variants of Hal{\'a}sz--Montgomery type mean value theorems that are tailored for character sums supported on numbers without small prime factors (for the usual Hal{\'a}sz--Montgomery type mean value theorem, see e.g.~\cite[Theorem 9.14]{iw-kow}). Using different $r$ in Burgess' bound one would obtain different results. A similar result can also be found from~\cite[Theorem 3]{Puchta}.

\begin{lemma}
\label{le:Hal-Mon}
Let $q \in \mathbb{N}$ and let $\chi_1, \dotsc, \chi_R$ be distinct Dirichlet characters of modulus $q$. Let $\varepsilon>0$ and $C \geq 1$ be fixed. Then the following hold.
\begin{enumerate}[(i)]
\item Let $N \in [q^\varepsilon, q^C]$. Then, for any complex coefficients $a_n$,
\[
\sum_{j = 1}^R \left|\sum_{\substack{n \leq N \\ (n, P(q^\varepsilon)) = 1}} a_n \chi_j(n)\right|^2 \ll \left(\frac{N}{\log q} + q^{\frac{1}{2}+2\varepsilon} R\right) \sum_{\substack{n \leq N \\ (n, P(q^\varepsilon)) = 1}} |a_n|^2.
\]
\item Let $q^C \geq N_2 \geq N_1 \geq q^{2\varepsilon}$. Then, for any complex coefficients $a_n$,
\[
\sum_{j = 1}^R \left|\sum_{\substack{N_1 < n \leq N_2 \\ (n, P(q^\varepsilon)) = 1}} \frac{a_n}{n} \chi_j(n)\right|^2 \ll \left(1 + \frac{q^{1/9+3\varepsilon}}{N_1^{1/3}} R\right) \sum_{\substack{N_1 < n \leq N_2 \\ (n, P(q^\varepsilon)) = 1}} \frac{|a_n|^2}{n}. 
\]
\end{enumerate}
\end{lemma}

\begin{proof}
We only prove the second claim, the first claim follows similarly (using non-logarithmic averages and taking $r=1$ instead of $r=3$ in Burgess' bound). By the duality principle (see e.g.~\cite[Section 7.1]{iw-kow}), it suffices to show that, for any complex coefficients $c_j$,
\begin{equation}
\label{eq:afterdual}
\sum_{\substack{N_1 < n \leq N_2 \\ (n, P(q^\varepsilon)) = 1}} \frac{1}{n} \left|\sum_{j = 1}^R c_j \chi_j(n)\right|^2 \ll \left(1 + \frac{q^{1/9+2\varepsilon}}{N_1^{1/3}} R\right) \sum_{\substack{j=1}}^R |c_j|^2. 
\end{equation}
Let $\lambda_d^+$ be as in Lemma~\ref{le:linear} with sifting parameter $D = q^\varepsilon$ and $s=1$ (so that $z = D$). Then the left-hand side of~\eqref{eq:afterdual} is by Lemma~\ref{le:linear}(ii)
\begin{align*}
&\leq \sum_{\substack{N_1 < n \leq N_2}} \frac{1}{n} \sum_{d \mid n} \lambda_d^+ \left|\sum_{j = 1}^R c_j \chi_j(n)\right|^2 = \sum_{j, k = 1}^R c_j \overline{c_k}  \sum_{d \leq q^\varepsilon} \frac{\lambda_d^+}{d} \chi_j(d) \overline{\chi_k(d)} \sum_{\substack{N_1/d < n \leq N_2/d}} \frac{\chi_j(n)\overline{\chi_k(n)}}{n}.
\end{align*}
The diagonal terms with $j = k$ contribute by~\eqref{eq:Simplegcd}, partial summation, and Lemma~\ref{le:linear}(iii)
\begin{align*}
\sum_{j=1}^R |c_j|^2 \sum_{\substack{d \leq q^\varepsilon \\ (d, q) = 1}} \frac{\lambda_d^+}{d} \sum_{\substack{N_1/d < n \leq N_2/d \\ (n, q) = 1}} \frac{1}{n} &= \sum_{j=1}^R |c_j|^2 \sum_{\substack{d \leq q^\varepsilon \\ (d, q) = 1}} \frac{\lambda_d^+}{d} \left(\frac{\varphi(q)}{q} \log \frac{N_2}{N_1} + O\left(\frac{\tau(q) d}{N_1}\right)\right) \\
&\ll \sum_{j=1}^R |c_j|^2. 
\end{align*}

On the other hand, the non-diagonal terms with $j \neq k$ contribute
\begin{align*}
\ll \sum_{\substack{j, k = 1 \\ j \neq k}}^R |c_j c_k| \sum_{d \leq q^\varepsilon} \frac{1}{d}\left|\sum_{\substack{N_1/d < n \leq N_2/d}} \frac{\chi_j(n)\overline{\chi_k(n)}}{n}\right|.
\end{align*}
By partial summation and the Burgess bound (Lemma~\ref{le:Burgess}) with $r = 3$, this is
\[
\ll R\sum_{j = 1}^R |c_j|^2 \sum_{d \leq q^\varepsilon} \frac{1}{d} \frac{q^{1/9+\varepsilon}}{(N_1/d)^{1/3}},
\]
and the claim follows.
\end{proof}

Lemma~\ref{le:Hal-MonL1chi} below will provide another Hal{\'a}sz--Montgomery type mean value theorem. 

\begin{remark}
\label{rem:Lrbound}
It would be possible to combine the previous three lemmas to obtain bounds for $r$th moments of character sums over primes with real $r > 2$ in the spirit of what is often done in the transference principle. However we obtain quantitatively stronger results by applying these lemmas directly in the proof of Proposition~\ref{prop:transfer} below.
\end{remark}

\section{A multiplicative transference principle}
\subsection{A dense model theorem}
Recall the notation 
\[
[N]_q = \{n \in \mathbb{N} \colon n\leq N \text{ and } (n, q) = 1\}.
\]
The following proposition is a dense model theorem which gives, for an unbounded function $f: [N]_q \to \mathbb{R}_{\geq 0}$ which is majorized by a pseudorandom measure, a model function $g : \mathbb{Z}_q^\times \to \mathbb{R}_{\geq 0}$ that is bounded and such that character sums of $f$ and $g$ behave similarly. 

Such results originate from the work of Green~\cite{green-annals}. The following proposition should be compared with e.g.~\cite[Proposition 5.1]{green-restriction},~\cite[Theorem 2.1]{prendiville-survey} or~\cite[Theorem 2.3]{matomaki-shao}.
\begin{proposition}[A multiplicative dense model theorem]
\label{prop:f=g+h}
Let $q \in \mathbb{N}$ and $N\geq 2$. Let $r > 1$ be fixed. Let $\eta, \varepsilon \in (0,1), C\geq 1$, and let 
\begin{equation}
\label{eq:deltabound}
\delta \in \left(\left(\frac{10 C r  \log \log q}{\varepsilon \log q}\right)^{1/r}, \frac{1}{10}\right).
\end{equation}
Let $f \colon [N]_q \to \mathbb{R}_{\geq 0}$ satisfy:
\begin{enumerate}[({A}1)]
\item There exists a majorant function $\nu \colon [N]_q \to \mathbb{R}_{\geq 0}$ such that $f(n) \leq \nu(n)$ for every $n \in [N]_q$, 
\[
\left|\mathbb{E}_{n \in [N]_q} \nu(n) - 1\right| \leq \eta, \quad \text{and} \quad \max_{\chi \neq \chi_0} \left|\mathbb{E}_{n \in [N]_q} \nu(n)\overline{\chi}(n) \right| \ll q^{-\varepsilon}.
\]
\item There are at most $C\delta^{-r}$ characters $\chi \pmod{q}$ such that
\[
\left|\mathbb{E}_{n \in [N]_q} f(n)\overline{\chi}(n) \right| \geq \delta.
\]
\end{enumerate}
Then there exists a function $g \colon \mathbb{Z}_q^\times \to \mathbb{R}_{\geq 0}$ with the following properties.
\begin{enumerate}[(i)]
\item For every $a \in \mathbb{Z}_q^{\times}$, we have 
\[
0 \leq g(a) \leq 1+\eta + O(q^{-\varepsilon/2}).
\]
\item We have, for any $\chi \pmod{q}$,
\[
\left|\mathbb{E}_{n \in [N]_q} f(n) \overline{\chi}(n) - \mathbb{E}_{a \in \mathbb{Z}_q^\times} g(a) \overline{\chi}(a)\right| \leq \delta.
\]
\item We have, for any $\chi \pmod{q}$,
\[
\left|\mathbb{E}_{a \in \mathbb{Z}_q^\times} g(a) \overline{\chi}(a)\right| \leq |\mathbb{E}_{n \in [N]_q} f(n) \overline{\chi}(n)|
\]
and
\[
\left|\mathbb{E}_{n \in [N]_q} f(n) \overline{\chi}(n) - \mathbb{E}_{a \in \mathbb{Z}_q^\times} g(a) \overline{\chi}(a)\right| \leq |\mathbb{E}_{n \in [N]_q} f(n) \overline{\chi}(n)|.
\]
\item We have $\mathbb{E}_{a \in \mathbb{Z}_q^{\times}} g(a) = \mathbb{E}_{n \in [N]_q} f(n)$.
\item
Let $H \leq \mathbb{Z}_q^{\times}$ be a subgroup of index $Y < 1/(2\delta)$. Then, for any $b \in \mathbb{Z}_q^{\times}$, we have
\[
\mathbb{E}_{\substack{n \in [N]_q}} f(n) \mathbf{1}_{n \in bH} = \mathbb{E}_{a \in \mathbb{Z}_q^\times} g(a) \mathbf{1}_{a \in bH} + O(\delta).
\]
\end{enumerate}
\end{proposition}

\begin{proof}
Let
\[
T \coloneqq \left\{\chi \pmod{q} \colon \left|\mathbb{E}_{n \in [N]_q} f(n)\overline{\chi}(n) \right| \geq \delta\right\}
\]
be the large spectrum of $f$, and define the multiplicative Bohr set
\[
B \coloneqq \{b \in \mathbb{Z}_q^{\times} \colon |\chi(b)-1| \leq \delta/5 \text{ for all $\chi \in T$}\}.
\]

By (A2) we have $|T| \leq C \delta^{-r}$. Let us bound $|B|$ from below similarly to~\cite[Proof of Lemma 4.20]{tao-vu}. Write $T = \{\chi_1, \dotsc, \chi_k\}$, where $\chi_j$ are distinct and $k \leq C \delta^{-r}$. By the pigeonhole principle, there exists a $k$-tuple $(\xi_1, \dotsc, \xi_k)\in \mathbb{C}^k$ with $|\xi_i| = 1$ for every $i$ such that, writing
\[
A = \{a \in \mathbb{Z}_q^{\times} \colon |\chi_j(a) - \xi_j| \leq \delta/10 \text{ for every $j = 1, \dotsc, k$}\},
\] 
we have
\[
|A| \geq (100/\delta+1)^{-k} \varphi(q).
\]
But now if $x, y \in A$, then $x/y \in B$. Hence, using~\eqref{eq:deltabound},
\begin{equation}
\label{eq:Blow}
|B| \geq |A| \geq \varphi(q) \exp(-C\delta^{-r} \log(100/\delta+1)) \geq q^{1-\varepsilon/2}.
\end{equation}

We define $g\colon \mathbb{Z}_q^\times \to \mathbb{R}$ by
\[
g(a) \coloneqq \varphi(q) \mathbb{E}_{b_1, b_2 \in B} \mathbb{E}_{n \in [N]_q} f(n) \mathbf{1}_{n \equiv a b_1 b_2^{-1} \pmod{q}}.
\]
Let us now verify the claims (i)--(v).

\textbf{Claim (i):} The lower bound $g(a) \geq 0$ is immediate from the definition of $g$. By the definition of $g$, assumption (A1) and orthogonality of characters,
\begin{align*}
g(a) &\leq \varphi(q) \mathbb{E}_{b_1, b_2 \in B} \mathbb{E}_{n \in [N]_q} \nu(n) \mathbf{1}_{n \equiv ab_1b_2^{-1} \pmod{q}}\\
&= \sum_{\chi \pmod{q}} \mathbb{E}_{n \in [N]_q} \nu(n) \overline{\chi}(n) \chi(a) \left|\mathbb{E}_{b \in B} \chi(b)\right|^2.
\end{align*}
Separating the contribution of the principal character, we obtain
\begin{align*}
g(a) &\leq \mathbb{E}_{n \in [N]_q} \nu(n) + O\left(\max_{\chi \neq \chi_0} \left|\mathbb{E}_{n \in [N]_q} \nu(n) \overline{\chi}(n)\right| \sum_{\chi \pmod{q}} \left|\mathbb{E}_{b \in B} \chi(b)\right|^2 \right).
\end{align*}
By the assumption (A1), Lemma~\ref{le:MVT}, and~\eqref{eq:Blow}, we obtain 
\[
g(a) \leq 1+\eta + O\left(\frac{q^{1-\varepsilon}}{|B|}\right) \leq 1+\eta + O(q^{-\varepsilon/2}).
\]

\textbf{Claim (ii):} Changing the order of summation and using the orthogonality of characters, we see that
\begin{equation}
\label{eq:gFTfFTchi}
\begin{split}
\mathbb{E}_{a \in \mathbb{Z}_q^\times} g(a) \overline{\chi}(a) &= \mathbb{E}_{b_1, b_2 \in B} \mathbb{E}_{a \in \mathbb{Z}_q^{\times}} \overline{\chi}(a) \varphi(q) \mathbb{E}_{n \in [N]_q} f(n) \mathbf{1}_{n \equiv a b_1 b_2^{-1} \pmod{q}} \\
& = \mathbb{E}_{b_1, b_2 \in B} \chi(b_1) \overline{\chi}(b_2) \mathbb{E}_{a \in \mathbb{Z}_q^{\times}} \overline{\chi}(a) \chi(a) \mathbb{E}_{n \in [N]_q} f(n) \overline{\chi}(n) \\
&= \left|\mathbb{E}_{b \in B} \chi(b)\right|^2 \mathbb{E}_{n \in [N]_q} f(n) \overline{\chi}(n).
\end{split}
\end{equation}
Hence
\begin{equation}
\label{eq:hFT}
\mathbb{E}_{n \in [N]_q} f(n) \overline{\chi}(n) - \mathbb{E}_{a \in \mathbb{Z}_q^\times} g(a) \overline{\chi}(a) = \mathbb{E}_{n \in [N]_q} f(n) \overline{\chi}(n) \left(1-\left|\mathbb{E}_{b \in B} \chi(b)\right|^2\right). 
\end{equation}

If now $\chi \not \in T$, then by~\eqref{eq:hFT} and the definition of $T$, we have
\[
\left|\mathbb{E}_{n \in [N]_q} f(n) \overline{\chi}(n) - \mathbb{E}_{a \in \mathbb{Z}_q^\times} g(a) \overline{\chi}(a)\right| \leq \left|\mathbb{E}_{n \in [N]_q} f(n) \overline{\chi}(n)\right| \leq \delta.
\]
On the other hand, if $\chi \in T$, then for every $b \in B$ we have $|\chi(b)-1| \leq \delta/5$ and thus by~\eqref{eq:hFT} and (A1) we obtain
\begin{align*}
\left|\mathbb{E}_{n \in [N]_q} f(n) \overline{\chi}(n) - \mathbb{E}_{a \in \mathbb{Z}_q^\times} g(a) \overline{\chi}(a)\right| &\leq \mathbb{E}_{n \in [N]_q} \nu(n) \left|1-\left|\mathbb{E}_{b \in B} \chi(b)\right|^2\right| \\
& \leq (1+\eta) \left(\left(1+\frac{\delta}{5}\right)^2-1\right).
\end{align*}
Since by assumptions of the proposition $\eta < 1$ and $\delta < 1/10$, the last expression is $\leq \delta$ and claim (ii) follows.

\textbf{Claim (iii):} This is immediate from~\eqref{eq:gFTfFTchi} and~\eqref{eq:hFT}.

\textbf{Claim (iv):} This follows from~\eqref{eq:gFTfFTchi} with $\chi = \chi_0$.

\textbf{Claim (v):} Let
\[
D = \{ \chi \pmod{q} \colon \chi(n) = 1 \text{ for every $n \in H$}\}.
\]
By~\eqref{eq:incoset}
\begin{equation}
\label{eq:cosdif}
\mathbb{E}_{\substack{n \in [N]_q}} f(n) \mathbf{1}_{n \in bH} - \mathbb{E}_{a \in \mathbb{Z}_q^\times} g(a) \mathbf{1}_{a \in bH} = \frac{1}{Y} \sum_{\chi \in D} \chi(b) \left(\mathbb{E}_{\substack{n \in [N]_q}} f(n) \overline{\chi}(n) - \mathbb{E}_{a \in \mathbb{Z}_q^\times} g(a) \overline{\chi}(a)\right)
\end{equation}

If now $\chi \not \in T$, then by (iii) and the definition of $T$ we have 
\[
\left|\mathbb{E}_{a \in \mathbb{Z}_q^{\times}}  g(a) \overline{\chi}(a) \right| \leq \left|\mathbb{E}_{n \in [N]_q} f(n) \overline{\chi}(n) \right| \leq \delta.
\]
On the other hand if $\chi \in T$, then (since $\delta < 1/(2Y)$ and $\chi$ takes values in $Y$th roots of unity) by the definition of $B$ we have $\chi(b) = 1$ for every $b \in B$ and so by~\eqref{eq:gFTfFTchi} we have
\[
\mathbb{E}_{a \in \mathbb{Z}_q^{\times}} g(a)\overline{\chi}(a)= \mathbb{E}_{n \in [N]_q} f(n) \overline{\chi}(n).
\]

Consequently, in any case 
\[
\mathbb{E}_{a \in \mathbb{Z}_q^{\times}} g(a)\overline{\chi}(a)= \mathbb{E}_{n \in [N]_q} f(n) \overline{\chi}(n) + O(\delta)
\]
and the claim (v) follows from inserting this into~\eqref{eq:cosdif}.
\end{proof}

\section{Applying the transference principle}
\label{se:applytrans}
Let $\varepsilon > 0$, $N \geq q^{1/2}$, and define 
\begin{equation}
\label{eq:thetaDef}
\theta = \theta(N, q) \coloneqq 
\begin{cases}
1-\varepsilon-\frac{\log q}{4\log N} & \text{if $q$ is cube-free;} \\
1-\varepsilon-\frac{\log q}{3\log N} & \text{otherwise},
\end{cases}
\end{equation}
and
\begin{equation}
\label{eq:theta0Def}
\theta_0 = \theta_0(N, q) \coloneqq 1-\varepsilon-\frac{\log q}{4\log N}.
\end{equation}

\begin{remark}\label{rmk:burgess} The only property of $\theta(N,q)$ needed in what follows is that 
\begin{align}\label{eq:burgessbound2}
\max_{\substack{\chi\pmod q\\\chi\neq \chi_0}} \max_{N^{1-\theta} \leq y \leq N} \frac{1}{y}\left|\sum_{n\leq y}\chi(n)\right|\ll N^{-\varepsilon_0}    
\end{align}
for some $\varepsilon_0>0$. Similarly, $\theta_0(N,q)$ is any exponent such that~\eqref{eq:burgessbound2} holds with $\theta_0$ in place of $\theta$ when the maximum is taken over characters $\chi\pmod q$ of bounded order.  
\end{remark}

The following proposition transfers the problem of studying $E_3(N)$ into a dense setup. We only need a qualitative result for $N \geq q$, but prove a more quantitative result in a larger range which might be helpful for further applications, or for working out a value of $\alpha$ for which $E_3(q^{1-\alpha}) = \mathbb{Z}_q^\times$ for every sufficiently large cube-free $q \in \mathbb{N}$.

\begin{proposition}[Conclusion of transference]\label{prop:transfer} Let $q \in \mathbb{N}$ be sufficiently large. Let $\kappa>0$ and $C \geq 1$ be fixed and let $N\in [q^{2/3+\kappa}, q^C]$. Let $\varepsilon>0$ be sufficiently small but fixed.  Then there exists a set $A\subseteq \mathbb{Z}_q^{\times}$ such that the following hold.
\begin{enumerate}[(i)]
\item We have
\begin{align*}
|A|\geq \left(\frac{\theta}{2}-\varepsilon\right)\varphi(q).
\end{align*}
\item For all but $O(\varphi(q) (\log q)^{-\varepsilon/5})$ values of $a\in \mathbb{Z}_q^{\times}$ we have
\begin{align}
\label{eq:doubleconv}
(\mathbf{1}_A*\mathbf{1}_A)(a) \gg \frac{\varphi(q)}{(\log q)^{1/2-\varepsilon}} \implies a \in E_2(N).
\end{align}
\item If $N \geq q^{4/5+\kappa}$, then
\begin{align}
\label{eq:tripconv}
(\mathbf{1}_A*\mathbf{1}_A*\mathbf{1}_A)(a) \gg \frac{\varphi(q)^2}{(\log q)^{1/2-\varepsilon}} \implies a \in E_3(N).
\end{align}
\item Assume that $N \geq q$ and let $\alpha_1 \in (2\varepsilon, 1]$. Write $L = \lfloor 1/\alpha_1\rfloor$. Then
\[
\sum_{\substack{a = p a_1 a_2 \\ q^{\alpha_1-\varepsilon} < p \leq q^{\alpha_1} \\ a_1, a_2 \in A}} \frac{1}{p} \gg \frac{\varphi(q)}{(\log q)^{\frac{1}{2L}-\varepsilon}} \implies a \in E_3(N).
\]
\item For any subgroup $H\leq \mathbb{Z}_q^{\times}$ of index $Y < \varepsilon^{-1/2}$, $A$ contains $>\varepsilon\varphi(q)$ elements from at least 
\[
\left\lceil \left(\frac{\theta_0}{2}-3\varepsilon^{1/2}\frac{\theta_0}{\theta}\right)Y\right\rceil 
\]
distinct cosets of $H$. 
\item For any subgroup $H\leq \mathbb{Z}_q^{\times}$ of index $Y < \varepsilon^{-1/2}$ and any $b\in \mathbb{Z}_q^{\times}$, we have
\begin{equation}
\label{eq:A'binprimes1}
|A\cap bH|\geq \left(\frac{\theta}{2}\cdot \frac{|E_1(N)\cap bH|}{N/\log N}-\frac{\varepsilon}{5Y}\right)\varphi(q).
\end{equation}
\end{enumerate}
\end{proposition}
Proposition~\ref{prop:transferQR} below adds one more conclusion to this proposition.

\begin{remark}
In the case where $q$ is cube-free and $N = q$, we obtain the bound $|A| \geq (3/8-3\varepsilon/2)\varphi(q)$ which is rather good. However, when $q$ is not cube-free, we only obtain that $|A| \geq (1/3-3\varepsilon/2)\varphi(q)$ which is problematic --- for instance $A$ could be stuck in a single coset of a subgroup of index $3$ and we could not deduce Theorem~\ref{thm:E2>1/2}(i). However, we can actually rule out this scenario by utilizing part (v) of the proposition.
\end{remark}

\begin{proof}[Proof of Proposition~\ref{prop:transfer}]
We let $D = N^\theta$ and $z = N^{\theta/3}$. We shall apply Proposition~\ref{prop:f=g+h} with $r = 2$,
\begin{equation}
\label{eq:deltadef}
\delta = \frac{1}{(\log q)^{1/2-\varepsilon/2}},
\end{equation}
and with $f, \nu \colon [N]_q \to \mathbb{R}_{\geq 0}$ defined by
\begin{equation}
\label{eq:fDef}
f(n) \coloneqq \frac{\theta}{2} \cdot \frac{\varphi(q)}{q} \log N \cdot \mathbf{1}_{n \in \mathbb{P}} \mathbf{1}_{n \geq z}
\end{equation}
and
\begin{equation}
\label{eq:nuDef}
\nu(n) \coloneqq \frac{\theta}{2} \cdot \frac{\varphi(q)}{q} \log N \cdot \sum_{\substack{d \mid n \\ d \leq D}} \lambda_d^+,
\end{equation}
where $\lambda_d^+$ are the upper bound linear sieve weights from Lemma~\ref{le:linear} with $D$ and $z$ as above (and $s = 3$), so that $f(n) \leq \nu(n)$ for every $n \in [N]_q$.

Now $f$ satisfies Proposition~\ref{prop:f=g+h}(A2) for some $C = O(1)$ by Lemma~\ref{le:Hal-Mon}(i). Let us confirm that $\nu$ satisfies Proposition~\ref{prop:f=g+h}(A1) for some $\eta=o(1)$. For any character $\chi \pmod{q}$, we have
\begin{align}
\label{eq:nuFT}
\mathbb{E}_{n \in [N]_q} \nu(n) \overline{\chi}(n) 
&=  \frac{\theta}{2} \cdot \frac{\varphi(q)}{q} \log N \cdot \mathbb{E}_{\substack{n \in [N]_q}} \overline{\chi}(n) \sum_{\substack{d \mid n \\ d \leq D}} \lambda_d^+ \\
\nonumber
&= \frac{\theta}{2} \cdot \frac{\varphi(q)}{q} \log N \sum_{d\leq D} \lambda_d^+ \overline{\chi}(d) \frac{1}{|[N]_q|} \sum_{\substack{m \leq N/d}} \overline{\chi}(m).
\end{align}
By the choice of $\theta$, here $N/d \geq q^{1/4+\varepsilon/4}$ when $q$ is cube-free and $N/d \geq q^{1/3+\varepsilon/4}$ otherwise. Hence when $\chi \neq \chi_0$, the Burgess bound (Lemma~\ref{le:Burgess}) gives that the innermost sum is $O(q^{-2\varepsilon_0}N/d)$ for some $\varepsilon_0=\varepsilon_0(\varepsilon) > 0$ and thus we have that, for any $\chi \neq \chi_0$,
\[
|\widehat{\nu}(\chi)| \ll q^{1-\varepsilon_0}.
\]

On the other hand, by~\eqref{eq:nuFT} and~\eqref{eq:Simplegcd},
\begin{equation}
\label{eq:nuFTchi0}
\begin{aligned}
\mathbb{E}_{n \in [N]_q} \nu(n) &= \mathbb{E}_{n \in [N]_q} \nu(n)\overline{\chi_0}(n) =  \frac{\theta}{2} \cdot \frac{\varphi(q)}{q} \cdot \log N \sum_{\substack{d \leq D \\ (d,q) = 1}} \lambda_d^+ \frac{1}{|[N]_q|} \sum_{\substack{n \leq N/d \\ (n, q) = 1}} 1 \\
&= \frac{\theta}{2} \cdot \frac{\varphi(q)}{q} \cdot \log N \frac{1}{N\frac{\varphi(q)}{q}+O(\tau(q))} \sum_{\substack{d \leq D \\ (d,q) = 1}} \lambda_d^+ \left(\frac{\varphi(q)N}{qd} + O(\tau(q))\right) \\
&= \frac{\theta}{2} \cdot \frac{\varphi(q)}{q} \log N \sum_{\substack{d \leq D \\ (d,q) = 1}} \frac{\lambda_d^+}{d} + O\left(\frac{Dq^{\varepsilon}\log N}{N}\right).
\end{aligned}
\end{equation}
By Lemma~\ref{le:linear}(iii) and Mertens' theorem,
\begin{equation}
\label{eq:lam+eval}
\sum_{\substack{d \leq D \\ (d,q) = 1}} \frac{\lambda_d^+}{d} = (2e^\gamma+o(1)) \frac{\log z}{\log D} \prod_{\substack{p < z \\ p \nmid q}} \left(1-\frac{1}{p}\right) = (1+o(1)) \frac{2}{\theta \log N} \cdot \frac{q}{\varphi(q)}
\end{equation}
and thus
\[
\mathbb{E}_{n \in [N]_q} \nu(n)  = 1+o(1)
\]
and Proposition~\ref{prop:f=g+h}(A1) holds for some $\eta = o(1)$.

Having established the assumptions, we may now apply Proposition~\ref{prop:f=g+h} with $r = 2$ and $\delta, f$ and $\nu$ as in~\eqref{eq:deltadef},~\eqref{eq:fDef} and~\eqref{eq:nuDef}. Hence we can find a function $g \colon \mathbb{Z}_q^\times \to [0, 1+o(1)]$ obeying Proposition~\ref{prop:f=g+h}(i)--(v). By Proposition~\ref{prop:f=g+h}(iv) and the prime number theorem,
\begin{equation}
\label{eq:fgav}
\mathbb{E}_{a \in \mathbb{Z}_q^\times} g(a) = \mathbb{E}_{n \in [N]_q} f(n) = \frac{\theta}{2} \cdot \frac{\varphi(q)}{q} \log N \cdot \mathbb{E}_{n \in [N]_q} \mathbf{1}_{n \in \mathbb{P}} \mathbf{1}_{n \geq z} = \frac{\theta}{2} + O((\log q)^{-1}).
\end{equation}
We will constantly use the fact that, by the definition of $f$, the prime number theorem, and~\eqref{eq:Nqsize},
\begin{equation}
\label{eq:fMS}
\frac{1}{|[N]_q|^2} \sum_{n \in [N]_q} f(n)^2 \ll \frac{1}{|[N]_q|^2}  \left(\frac{\varphi(q)}{q}\right)^2 N \log N \ll \frac{\log N}{N}.
\end{equation}
Write
\begin{align}\label{eq:Adef}
A \coloneqq \{a \in \mathbb{Z}_q^{\times} \colon |g(a)| \geq \varepsilon/10\}.
\end{align}
Now we are ready to establish the claims (i)--(v) of Proposition~\ref{prop:transfer}.

\textbf{Claim (i):} By Proposition~\ref{prop:f=g+h}(i),
\[
\mathbb{E}_{a \in \mathbb{Z}_q^\times} g(a) = \frac{1}{\varphi(q)}\left(\sum_{a \in \mathbb{Z}_q^\times \setminus A} g(a) + \sum_{a \in A} g(a)\right) \leq \frac{\varepsilon}{10} + \frac{|A|}{\varphi(q)}(1+o(1)).
\]
Combining this with~\eqref{eq:fgav} we obtain
\begin{align}\label{eq:A_theta}
|A| \geq \left(\frac{\theta}{2}-\varepsilon\right)\varphi(q).
\end{align}

\textbf{Claim (iii):} Note that definitely $a \in E_3(N)$ if $(f \ast f \ast f)(n) > 0$ for some $n \equiv a \pmod{q}$. Now, by orthogonality of characters and Proposition~\ref{prop:f=g+h}(iii) we obtain that
\begin{align}
\begin{aligned}
\label{eq:fff-ggg}
&\frac{1}{|[N]_q|^3} \sum_{n \equiv a \pmod{q}} (f \ast f \ast f)(n) - \frac{1}{\varphi(q)^3} (g \ast g \ast g)(a) \\
&=\frac{1}{\varphi(q)} \sum_{\chi \pmod{q}} \left(\mathbb{E}_{n \in [N]_q} f(n) \overline{\chi}(n)\right)^3 \chi(a) - \frac{1}{\varphi(q)} \sum_{\chi \pmod{q}} \left(\mathbb{E}_{b \in \mathbb{Z}_q^\times} g(b) \overline{\chi}(n)\right)^3 \chi(a) \\
&= O\left(\frac{1}{\varphi(q)} \sum_{\chi \pmod{q}} \left|\mathbb{E}_{n \in [N]_q} f(n) \overline{\chi}(n)\right|^2 \left|\mathbb{E}_{n \in [N]_q} f(n) \overline{\chi}(n) - \mathbb{E}_{b \in \mathbb{Z}_q^\times} g(b) \overline{\chi}(b)\right|\right).
\end{aligned}
\end{align}
To bound the right-hand side, we split the characters into three sets
\begin{equation}
\label{eq:X1def}
\mathcal{X}_1 = \left\{\chi \pmod{q} \colon \left|\mathbb{E}_{n \in [N]_q} f(n) \overline{\chi}(n)\right| \leq \frac{1}{N^{1/4}}\right\},
\end{equation}
\begin{equation}
\label{eq:X2def}
\mathcal{X}_2 = \left\{\chi \pmod{q} \colon \frac{1}{N^{1/4}} < \left|\mathbb{E}_{n \in [N]_q} f(n) \overline{\chi}(n)\right| \leq \frac{q^{1/4+\kappa}}{N^{1/2}}\right\},
\end{equation}
and
\begin{equation}
\label{eq:X3def}
\mathcal{X}_3 = \left\{\chi \pmod{q} \colon \left|\mathbb{E}_{n \in [N]_q} f(n) \overline{\chi}(n)\right| > \frac{q^{1/4+\kappa}}{N^{1/2}}\right\}.
\end{equation}

By the definition of $\mathcal{X}_1$, Proposition~\ref{prop:f=g+h}(iii), Lemma~\ref{le:MVT} and~\eqref{eq:fMS}, we see that 
\begin{align}
\begin{aligned}
\label{eq:iiiX1}
&\frac{1}{\varphi(q)} \sum_{\chi \in \mathcal{X}_1} \left|\mathbb{E}_{n \in [N]_q} f(n) \overline{\chi}(n)\right|^2 \left|\mathbb{E}_{n \in [N]_q} f(n) \overline{\chi}(n) - \mathbb{E}_{b \in \mathbb{Z}_q^\times} g(b) \overline{\chi}(b)\right| \\
&\leq \frac{1}{\varphi(q)} \frac{1}{N^{1/4}} \sum_{\chi \pmod{q}} \left|\mathbb{E}_{n \in [N]_q} f(n) \overline{\chi}(n)\right|^2 \\
&\ll \frac{1}{\varphi(q)} \frac{1}{N^{1/4}} \left(N+\varphi(q)\right)\frac{\log N}{N} \ll \frac{q^{\kappa}}{N^{1/4} q} + \frac{q^{\kappa}}{N^{5/4}}.
\end{aligned}
\end{align}
This is $O(1/q^{1+\kappa/4})$ when $N \geq q^{4/5+\kappa}$.

On the other hand, by the definition of $\mathcal{X}_3$, Lemma~\ref{le:Hal-Mon}(i) with $\varepsilon = \kappa/2$ and~\eqref{eq:fMS},
\begin{align*}
\left(\frac{q^{1/4+\kappa}}{N^{1/2}}\right)^2 |\mathcal{X}_3| &\leq  \sum_{\chi \in \mathcal{X}_3} \left|\mathbb{E}_{n \in [N]_q} f(n) \overline{\chi}(n)\right|^2 \ll \left(\frac{N}{\log q} + q^{1/2+\kappa} |\mathcal{X}_3| \right) \frac{\log N}{N}.
\end{align*}
Now the second term cannot dominate and so 
\begin{equation}
\label{eq:fnsecondHM}
\sum_{\chi \in \mathcal{X}_3} \left|\mathbb{E}_{n \in [N]_q} f(n) \overline{\chi}(n)\right|^2 \ll 1.
\end{equation}
Combining this with Proposition~\ref{prop:f=g+h}(ii), we obtain 
\begin{equation}
\label{eq:iiiX3}
\frac{1}{\varphi(q)} \sum_{\chi \in \mathcal{X}_3} \left|\mathbb{E}_{n \in [N]_q} f(n) \overline{\chi}(n)\right|^2 \left|\mathbb{E}_{n \in [N]_q} f(n) \overline{\chi}(n) - \mathbb{E}_{b \in \mathbb{Z}_q^\times} g(b) \overline{\chi}(b)\right| \ll \frac{\delta}{\varphi(q)}.
\end{equation}

To deal with $\mathcal{X}_2$, we write, for $\delta_0 \in [N^{-1/4}, q^{1/4+\kappa}/N^{1/2}]$,
\begin{equation}
\label{eq:X2deldef}
\mathcal{X}_{2, \delta_0} := \left\{\chi \pmod{q} \colon \delta_0 < \left|\mathbb{E}_{n \in [N]_q} f(n) \overline{\chi}(n)\right| \leq 2\delta_0\right\}.
\end{equation}
If we can show that, for any $\delta_0 \in [N^{-1/4}, q^{1/4+\kappa}/N^{1/2}]$, we have
\begin{equation}
\label{eq:X2delbound}
|\mathcal{X}_{2, \delta_0}| \ll \frac{\delta_0^{-3}}{(\log q)^{10}}, 
\end{equation}
then we obtain using also Proposition~\ref{prop:f=g+h}(iii) that
\begin{align}
\begin{aligned}
\label{eq:iiiX2}
&\frac{1}{\varphi(q)} \sum_{\chi \in \mathcal{X}_2} \left|\mathbb{E}_{n \in [N]_q} f(n) \overline{\chi}(n)\right|^2 \left|\mathbb{E}_{n \in [N]_q} f(n) \overline{\chi}(n) - \mathbb{E}_{b \in \mathbb{Z}_q^\times} g(b) \overline{\chi}(b)\right| \\
&\ll \frac{1}{\varphi(q)}(\log q) \max_{\delta_0 \in [N^{-1/4}, q^{1/4+\kappa}/N^{1/2}]} \sum_{\chi \in \mathcal{X}_{2, \delta_0}} \left|\mathbb{E}_{n \in [N]_q} f(n) \overline{\chi}(n)\right|^3 \ll \frac{1}{\varphi(q) (\log q)^9}.
\end{aligned}
\end{align}
Now~\eqref{eq:X2delbound} follows for $\delta_0 \in [N^{-1/4}, q^{1/4+\kappa}/N^{1/2}]$ and $N \geq q^{4/5+\kappa}$ since Lemma~\ref{le:Huxley} implies that
\begin{equation}
\label{eq:X2delHux}
|\mathcal{X}_{2, \delta_0}| \ll \delta_0^{-2} (\log q)^7 + \delta_0^{-6} q N^{-2} (\log q)^{21}.
\end{equation}

Using~\eqref{eq:iiiX1},~\eqref{eq:iiiX3}, and~\eqref{eq:iiiX2} on the right-hand side of~\eqref{eq:fff-ggg}, we obtain
\begin{equation}
\label{eq:fff=ggg}
\frac{1}{|[N]_q|^3} \sum_{n \equiv a \pmod{q}} (f \ast f \ast f)(n) = \frac{1}{\varphi(q)^3} (g \ast g \ast g)(a) + O\left(\frac{\delta}{\varphi(q)}\right).
\end{equation}

By the definitions of $f$ and $\delta$ and~\eqref{eq:fff=ggg} we obtain that $a \in E_3(N)$ if
\[
(g \ast g \ast g)(a) \gg \frac{\varphi(q)^2}{(\log q)^{1/2-\varepsilon}}.
\]
But this obviously holds if $(\mathbf{1}_A \ast \mathbf{1}_A \ast \mathbf{1}_A)(a) \gg \varphi(q)^2 (\log q)^{-1/2+\varepsilon}$.

\textbf{Claim (ii):}
Now definitely $a \in E_2(N)$ if $(f \ast f)(n) > 0$ for some $n \equiv a \pmod{q}$. By the orthogonality of characters 
\begin{align}
\begin{aligned}
\label{eq:orthbin}
&\frac{1}{|[N]_q|^2} \sum_{n \equiv a \pmod{q}} (f \ast f)(n) - \frac{1}{\varphi(q)^2} (g \ast g)(a)\\
=& \frac{1}{\varphi(q)} \sum_{\chi \pmod{q}} \left(\left(\mathbb{E}_{n \in [N]_q} f(n) \overline{\chi}(n)\right)^2 - \left(\mathbb{E}_{b \in \mathbb{Z}_q^\times} g(b) \overline{\chi}(b)\right)^2\right) \chi(a).
\end{aligned}
\end{align}
By orthogonality of characters and Proposition~\ref{prop:f=g+h}(iii),
\begin{align}
\begin{aligned}
\label{eq:ff-gg}
&\sum_{a \in \mathbb{Z}_q^\times} \left|\frac{1}{\varphi(q)} \sum_{\chi \pmod{q}} \left(\left(\mathbb{E}_{n \in [N]_q} f(n) \overline{\chi}(n)\right)^2 - \left(\mathbb{E}_{b \in \mathbb{Z}_q^\times} g(b) \overline{\chi}(b)\right)^2\right) \chi(a)\right|^2 \\
&= \frac{1}{\varphi(q)} \sum_{\chi \pmod{q}} \left|\left(\mathbb{E}_{n \in [N]_q} f(n) \overline{\chi}(n)\right)^2 - \left(\mathbb{E}_{b \in \mathbb{Z}_q^\times} g(b) \overline{\chi}(b)\right)^2\right|^2 \\
&\ll \frac{1}{\varphi(q)} \sum_{\chi \pmod{q}} \left|\mathbb{E}_{n \in [N]_q} f(n) \overline{\chi}(n)\right|^2 \left|\mathbb{E}_{n \in [N]_q} f(n) \overline{\chi}(n) - \mathbb{E}_{b \in \mathbb{Z}_q^\times} g(b) \overline{\chi}(b)\right|^2.
\end{aligned}
\end{align}

We split the characters into three sets $\mathcal{X}_1 \cup \mathcal{X}_2 \cup \mathcal{X}_3$ as in~\eqref{eq:X1def},~\eqref{eq:X2def}, and~\eqref{eq:X3def}. By the definition of $\mathcal{X}_1$, Proposition~\ref{prop:f=g+h}(iii), Lemma~\ref{le:MVT}, and~\eqref{eq:fMS} we see that 
\begin{align}
\begin{aligned}
\label{eq:iiX1}
&\frac{1}{\varphi(q)} \sum_{\chi \in \mathcal{X}_1} \left|\mathbb{E}_{n \in [N]_q} f(n) \overline{\chi}(n)\right|^2 \left|\mathbb{E}_{n \in [N]_q} f(n) \overline{\chi}(n) - \mathbb{E}_{b \in \mathbb{Z}_q^\times} g(b) \overline{\chi}(b)\right|^2 \\
&\leq \left(\frac{1}{N^{1/4}}\right)^2 \frac{1}{\varphi(q)} \sum_{\chi \pmod{q}} \left|\mathbb{E}_{n \in [N]_q} f(n) \overline{\chi}(n)\right|^2 \\
&\ll \frac{1}{N^{1/2}\varphi(q)} (N+\varphi(q)) \frac{\log N}{N} \ll \frac{\log N}{N^{1/2} \varphi(q)} + \frac{\log N}{N^{3/2}}.
\end{aligned}
\end{align}
This is $O(1/q^{1+\kappa/2})$ when $N \geq q^{2/3+\kappa}$.
On the other hand, recall that by~\eqref{eq:fnsecondHM}
\[
\sum_{\chi \in \mathcal{X}_3} \left|\mathbb{E}_{n \in [N]_q} f(n) \overline{\chi}(n)\right|^2 \ll 1.
\]
Combining this with Proposition~\ref{prop:f=g+h}(ii), we obtain 
\begin{equation}
\label{eq:iiX3}
\frac{1}{\varphi(q)} \sum_{\chi \in \mathcal{X}_3} \left|\mathbb{E}_{n \in [N]_q} f(n) \overline{\chi}(n)\right|^2 \left|\mathbb{E}_{n \in [N]_q} f(n) \overline{\chi}(n) - \mathbb{E}_{b \in \mathbb{Z}_q^\times} g(b) \overline{\chi}(b)\right|^2 \ll \frac{\delta^2}{\varphi(q)}.
\end{equation}

To deal with $\mathcal{X}_2$, we recall the definition of $\mathcal{X}_{2, \delta_0}$ from~\eqref{eq:X2deldef}. If we can show that, for any $\delta_0 \in [N^{-1/4}, q^{1/4+\kappa}/N^{1/2}]$, we have
\begin{equation}
\label{eq:X2delboundbin}
|\mathcal{X}_{2, \delta_0}| \ll \frac{\delta_0^{-4}}{(\log q)^{10}}, 
\end{equation}
then we obtain using also Proposition~\ref{prop:f=g+h}(iii) that
\begin{align}
\label{eq:iiX2}
\frac{1}{\varphi(q)} \sum_{\chi \in \mathcal{X}_2} \left|\mathbb{E}_{n \in [N]_q} f(n) \overline{\chi}(n)\right|^2 \left|\mathbb{E}_{n \in [N]_q} f(n) \overline{\chi}(n) - \mathbb{E}_{b \in \mathbb{Z}_q^\times} g(b) \overline{\chi}(b)\right|^2 \ll \frac{1}{\varphi(q) (\log q)^9}.
\end{align}
Now~\eqref{eq:X2delboundbin} follows for $\delta_0 \in [N^{-1/4}, q^{1/4+\kappa}/N^{1/2}]$ and $N \geq q^{2/3+\kappa}$ from~\eqref{eq:X2delHux}. 

Combining~\eqref{eq:iiX1},~\eqref{eq:iiX3}, and~\eqref{eq:iiX2}, we obtain
\[
\frac{1}{\varphi(q)} \sum_{\chi \pmod{q}} \left|\mathbb{E}_{n \in [N]_q} f(n) \overline{\chi}(n)\right|^2 \left|\mathbb{E}_{n \in [N]_q} f(n) \overline{\chi}(n) - \mathbb{E}_{b \in \mathbb{Z}_q^\times} g(b) \overline{\chi}(b)\right|^2 \ll \frac{\delta^2}{\varphi(q)}.
\]
This together with~\eqref{eq:orthbin} and~\eqref{eq:ff-gg} implies
\begin{equation}
\label{eq:ff=gg}
\frac{1}{|[N]_q|^2} \sum_{n \equiv a \pmod{q}} (f \ast f)(n) = \frac{1}{\varphi(q)^2} (g \ast g)(a) + O\left(\frac{\delta^{1-\varepsilon/2}}{\varphi(q)}\right)
\end{equation}
for all but $O(\delta^{\varepsilon/2} \varphi(q))$ elements $a \in \mathbb{Z}_q^\times$. Now
\[
\frac{1}{\varphi(q)^2} (g \ast g)(a) + O\left(\frac{\delta^{1-\varepsilon/2}}{\varphi(q)}\right) \geq \frac{\varepsilon^2}{\varphi(q)^2} (\mathbf{1}_A \ast \mathbf{1}_A)(a) +  O\left(\frac{\delta^{1-\varepsilon/2}}{\varphi(q)}\right)
\]
and claim (ii) follows recalling the definition of $\delta$.

\textbf{Claim (iv):} Let $f_0 \colon [N]_q \to \mathbb{R}_{\geq 0}$ be defined through
\[
f_0(n) = \frac{\mathbf{1}_{n \in \mathbb{P} \cap (q^{\alpha_1-\varepsilon}, q^{\alpha_1}]}}{n}.
\]
Note that definitely $a \in E_3(N)$ if $(f \ast f \ast f_0)(n) > 0$ for some $n \equiv a \pmod{q}$. Now by orthogonality of characters and Proposition~\ref{prop:f=g+h}(iii) (abusing the notation to consider $f_0$ as a function from $\mathbb{Z}_q^\times$ when convolved with $g$),
\begin{align}
\begin{aligned}
\label{eq:ff0--ggf0}
&\frac{1}{|[N]_q|^2} \sum_{n \equiv a \pmod{q}} (f \ast f \ast f_0)(n) - \frac{1}{\varphi(q)^2} (g \ast g \ast f_0)(a) \\
&=\frac{1}{\varphi(q)} \sum_{\chi \pmod{q}} \left( \left(\mathbb{E}_{n \in [N]_q} f(n) \overline{\chi}(n)\right)^2 - \left(\mathbb{E}_{b \in \mathbb{Z}_q^\times} g(b) \overline{\chi}(b)\right)^2\right)\sum_{n \in [N]_q} f_0(n) \overline{\chi(n)} \chi(a) \\
&=O\bigg(\frac{1}{\varphi(q)} \sum_{\chi \pmod{q}} \left|\mathbb{E}_{n \in [N]_q} f(n) \overline{\chi}(n)\right| \bigg| \sum_{n \in [N]_q} f_0(n) \overline{\chi(n)}\bigg|\\
&\quad \quad \quad \quad\cdot\left|\mathbb{E}_{n \in [N]_q} f(n) \overline{\chi}(n) - \mathbb{E}_{b \in \mathbb{Z}_q^\times} g(b) \overline{\chi}(b)\right|\bigg).
\end{aligned}
\end{align}

When considering the right-hand side, we split the characters into two sets
\[
\mathcal{X}_1 = \left\{\chi \pmod{q} \colon \left|\sum_{n \in [N]_q} f_0(n) \overline{\chi}(n)\right| < (\log q)^{-10} \right\}
\]
and
\[
\mathcal{X}_2 = \left\{\chi \pmod{q} \colon \left|\sum_{n \in [N]_q} f_0(n) \overline{\chi}(n)\right| \geq (\log q)^{-10} \right\}.
\]
By the definition of $\mathcal{X}_1$, Proposition~\ref{prop:f=g+h}(iii), Lemma~\ref{le:MVT}, and~\eqref{eq:fMS}, we see that 
\begin{align}
\begin{aligned}
\label{eq:X1log}
&\frac{1}{\varphi(q)} \sum_{\chi \in \mathcal{X}_1} \left|\mathbb{E}_{n \in [N]_q} f(n) \overline{\chi}(n)\right| \left| \sum_{n \in [N]_q} f_0(n) \overline{\chi(n)}\right| \left|\mathbb{E}_{n \in [N]_q} f(n) \overline{\chi}(n) - \mathbb{E}_{b \in \mathbb{Z}_q^\times} g(b) \overline{\chi}(b)\right| \\
&\ll (\log q)^{-10} \frac{1}{\varphi(q)} \sum_{\chi \pmod{q}} \left|\mathbb{E}_{n \in [N]_q} f(n) \overline{\chi}(n)\right|^2 \\
&\ll   (\log q)^{-10} \frac{1}{\varphi(q)} (N+\varphi(q)) \frac{\log N}{N} \ll \frac{1}{\varphi(q) (\log q)^{9}}.
\end{aligned}
\end{align}
Recall that $L = \lfloor 1/\alpha_1 \rfloor$. By Proposition~\ref{prop:f=g+h}(ii,iii) and H\"older's inequality, we obtain 
\begin{align}
\begin{aligned}
\label{eq:X2f0}
&\frac{1}{\varphi(q)} \sum_{\chi \in \mathcal{X}_2} \left|\mathbb{E}_{n \in [N]_q} f(n) \overline{\chi}(n)\right|  \left|\sum_{n \in [N]_q} f_0(n) \overline{\chi}(n)\right| \left|\mathbb{E}_{n \in [N]_q} f(n) \overline{\chi}(n) - \mathbb{E}_{b \in \mathbb{Z}_q^\times} g(b) \overline{\chi}(b)\right| \\
&\ll \delta^{1/L} \frac{1}{\varphi(q)} \left(\sum_{\chi \in \mathcal{X}_2} \left|\sum_{n \in [N]_q} f_0(n) \overline{\chi}(n)\right|^{2L}\right)^{\frac{1}{2L}}  \left(\sum_{\chi \in \mathcal{X}_2}\left|\mathbb{E}_{n \in [N]_q} f(n) \overline{\chi}(n)\right|^2\right)^{1-\frac{1}{2L}}.
\end{aligned}
\end{align}
Recall we assume that $\varepsilon$ is sufficiently small. Then $[L(\alpha_1-\varepsilon), L\alpha_1] \subseteq [3/8, 1]$. We obtain by the definition of $\mathcal{X}_2$, Lemma~\ref{le:Hal-Mon}(ii), and Mertens' theorem that
\begin{align*}
\left((\log q)^{-10}\right)^{2L} |\mathcal{X}_2| &\leq  \sum_{\chi \in \mathcal{X}_2} \left|\sum_{n \in [N]_q} f_0(n) \overline{\chi}(n)\right|^{2L}\\
&= \sum_{\chi \in \mathcal{X}_2} \left|\sum_{q^{L(\alpha_1-\varepsilon)} < n \leq q^{L\alpha_1}} \left(\sum_{\substack{n = p_1 \dotsm p_L \\ q^{\alpha_1-\varepsilon} < p_1,\ldots,p_L \leq q^{\alpha_1}}} 1\right)\frac{\overline{\chi}(n)}{n} \right|^2 \\
&\leq \left(1 + q^{1/9+3\varepsilon-1/8}|\mathcal{X}_2|\right) \sum_{\substack{n \leq q \\ (n, P(q^{\alpha_1-\varepsilon})) = 1}} \frac{\tau_L(n)^2}{n} \\
& \ll (1+q^{-1/100}|\mathcal{X}_2|) \prod_{q^{\alpha_1-\varepsilon} \leq p \leq q} \left(1+\frac{L^2}{p}\right) \ll 1+q^{-1/100}|\mathcal{X}_2|.
\end{align*}
Now the second term cannot dominate, and so $|\mathcal{X}_2| \ll (\log q)^{20L}$ and
\begin{equation*}
\sum_{\chi \in \mathcal{X}_2} \left|\sum_{n \in [N]_q} f_0(n) \overline{\chi}(n)\right|^{2L} \ll 1.
\end{equation*}

Applying also Lemma~\ref{le:Hal-Mon}(i), we see that~\eqref{eq:X2f0} is $\ll \delta^{\frac{1}{L}}/\varphi(q)$. Combining this with~\eqref{eq:X1log} we obtain from~\eqref{eq:ff0--ggf0} that
\begin{equation}
\label{eq:fff0=ggf0delta}
\frac{1}{|[N]_q|^2} \sum_{n \equiv a \pmod{q}} (f \ast f \ast f_0)(n) = \frac{1}{\varphi(q)^2} (g \ast g \ast f_0)(a) + O\left(\frac{\delta^{\frac{1}{L}}}{\varphi(q)}\right),
\end{equation}
where we still have abused the notation by considering on the right-hand side $f_0$ as a function from $\mathbb{Z}_q^\times$.

By the definitions of $f$ and $f_0$ and~\eqref{eq:fff0=ggf0delta} we obtain that $a \in E_3(N)$ if
\[
(g \ast g \ast f_0)(a) \gg \varphi(q) \delta^{1/L-\varepsilon/4},
\]
and the claim follows by the definitions of $A$ and $\delta$.

\begin{remark}
The reason we needed to use logarithmic averaging in Proposition~\ref{prop:transfer}(iv) is that when $\alpha_1 < 1/3+\varepsilon$, we cannot apply Hal\'asz--Montgomery type results without losses without logarithmic averaging. In this range, in order to make the Burgess bound applicable, we need to take higher than second moment of the character sum, and we cannot use a sieve to find a correct order of magnitude sieve majorant for example for the characteristic function of integers of the form $n = p_1 p_2$ with $p_j \sim q^{\alpha_1}$.
\end{remark}

\textbf{Claim (v):} Let $H \leq \mathbb{Z}_q^\times$ be a subgroup of index $Y < \varepsilon^{-1/2}$, and let $b_1H, \dotsc, b_KH$ be all the distinct cosets of $H$ from which $A$ contains at least $\varepsilon \varphi(q)$ elements, and let $c_1 H, \dotsc, c_{Y-K} H$ be the remaining distinct cosets. Thanks to Proposition~\ref{prop:f=g+h}(v), the definition of $f$, and Lemma~\ref{le:P2}, we have, for any $j \in \{1, \dotsc, K\}$
\begin{align}
\begin{aligned}
\label{eq:gf}
\frac{1}{\varphi(q)} \sum_{a \in b_j H} g(a) &= \mathbb{E}_{n \in [N]_q} \mathbf{1}_{n \in b_j H} f(n) + O(\delta) \\
&= \frac{\theta}{2} \cdot \frac{\varphi(q)}{q} \log N \frac{1}{|[N]_q|} \sum_{\substack{z \leq p \leq N \\ p \in b_jH}} 1 + O(\delta) \leq (1+o(1)) \frac{\theta}{Y \theta_0},
\end{aligned}
\end{align}
recalling that $\delta=o(1)$. On the other hand, for any $j \in \{1, \dotsc, Y-K\}$ we have by Proposition~\ref{prop:f=g+h}(i) and~\eqref{eq:Adef} (using also $\varepsilon\leq \varepsilon^{1/2}/Y$)
\begin{align*}
\frac{1}{\varphi(q)} \sum_{a \in c_j H} g(a) &= \frac{1}{\varphi(q)} \sum_{a \in A \cap c_jH} g(a) + \frac{1}{\varphi(q)} \sum_{a \in A^c \cap c_j H} g(a) \\
& \leq \frac{1}{\varphi(q)} \cdot \varepsilon \varphi(q) \cdot \left(1+\frac{\varepsilon}{10}\right) + \frac{1}{\varphi(q)} \cdot \frac{\varphi(q)}{Y} \cdot \frac{\varepsilon}{10} \leq \frac{2\varepsilon^{1/2}}{Y}.
\end{align*}
Combining these with~\eqref{eq:fgav}, we obtain that
\[
\frac{\theta}{2} \leq (1+o(1))K \frac{\theta}{Y \theta_0} + 2\varepsilon^{1/2} \varphi(q) + o(1),
\] 
and the claim follows.

\textbf{Claim (vi):} The definition of $f$, Proposition~\ref{prop:f=g+h}(v, i), and the definition of $A$ imply that
\begin{align*}
\frac{\theta}{2} \frac{|E_1(N) \cap bH|}{N/\log N} &= \frac{1}{\varphi(q)} \sum_{a \in bH} g(a) + o(1) = \frac{1}{\varphi(q)} \sum_{a \in A \cap bH} g(a) + \frac{1}{\varphi(q)} \sum_{a \in A^c \cap bH} g(a) + o(1) \\
&\leq \left(1+\frac{\varepsilon}{20}\right)\frac{|A \cap bH|}{\varphi(q)} + \frac{\varepsilon}{10Y} +o(1)\leq \frac{|A \cap bH|}{\varphi(q)} + \frac{\varepsilon}{5Y},
\end{align*}
and~\eqref{eq:A'binprimes1} follows.
\end{proof}

\section{Products sets}

\subsection{Reduction from popular products to product sets}

In the notation of Proposition~\ref{prop:transfer}, we have reduced the study of $E_3(q)$ to getting a lower bound of the form $(\mathbf{1}_{A}*\mathbf{1}_{A}*\mathbf{1}_{A})(a) \gg \varphi(q)^2$ for all $a\in \mathbb{Z}_q^{\times}$. If instead it sufficed just to show that $a \in A \cdot A \cdot A$, we could apply Kneser's theorem (Lemma~\ref{le:Kneser}). However, we need that $a$ has many representations as a product of three elements of $A$. Fortunately, the following lemma which is a quick consequence of work of Grynkiewicz~\cite{Grynkiewicz} allows us to essentially reduce to studying $A \cdot A \cdot A$.

\begin{lemma}
\label{le:popularkneser}
Let $t\geq u\geq 1$ be integers. Let $A,B$ be finite subsets of a finite abelian group $G$ with $|A|,|B|\geq t$. Then at least one of the following holds.
\begin{enumerate}[(a)]
\item We have \begin{align*}
(\mathbf{1}_{A}*\mathbf{1}_{B})(a)\geq u    
\end{align*}
for at least
\begin{align*}
|A|+|B| - 2 t -\frac{u|G|}{t}   \end{align*}
values of $a\in G$.
\item There exist subsets $A'\subseteq A, B'\subseteq B$ with
\begin{align*}
|A\setminus A'|+|B\setminus B'|\leq t-1    
\end{align*}
such that  
\[
(\mathbf{1}_{A}*\mathbf{1}_{B})(a)\geq t
\]
for every $a \in A' \cdot B'$.
\end{enumerate}
\end{lemma}

\begin{proof}
Let $N_i(A,B)$ denote the set of $c \in G$ such that $c$ has at least $i$ representations as $ab$ with $a\in A, b\in B$. Then by~\cite[Theorem 1.1]{Grynkiewicz} either (b) holds, or we have
 \begin{align}\label{eq:Ni}
\sum_{i=1}^t|N_i(A,B)|\geq t(|A|+|B|-2t)+1.
 \end{align}
If~\eqref{eq:Ni} holds, then we trivially also have
 \begin{align*}
\sum_{i=u}^t|N_i(A,B)|\geq t(|A|+|B|-2t)- u |G|.     
 \end{align*}
 Since $|N_u(A,B)|\geq |N_t(A,B)|$ for $t\geq u$, we conclude that
 \begin{align*}
|N_u(A,B)|\geq \frac{1}{t} \sum_{i = u}^t |N_i(A, B)| \geq |A|+|B|-2t-\frac{u|G|}{t}.     
 \end{align*}
The claim (a) follows, since for each $a\in N_u(A,B)$ we have $(\mathbf{1}_{A}*\mathbf{1}_{B})(a)\geq u$.
\end{proof}

Next we combine the previous lemma with Proposition~\ref{prop:transfer} to obtain criteria about when $a \in E_j(N)$.
\begin{proposition}[Criteria on sizes of $E_2(N)$ and $E_3(N)$]
\label{prop:A'binter}
Let $\kappa, \varepsilon>0$ and $C \geq 1$ be fixed. Let $q \in \mathbb{N}$ be large in terms of $\kappa, \varepsilon$, and let $N \in [q^{2/3+\kappa}, q^C]$. Let $\theta$ and $\theta_0$ be as in~\eqref{eq:thetaDef} and~\eqref{eq:theta0Def}. Then at least one of the following holds:
\begin{enumerate}[(a)]
\item We have $|E_2(N)| \geq (\theta - 2\varepsilon) \varphi(q)$. If $\theta > 2/3+3\varepsilon$, then we also have $E_3(N) = \mathbb{Z}_q^\times$. 
\item There exist sets $A', B' \subseteq \mathbb{Z}_q^\times$ such that the following hold. 
\begin{itemize}
\item[(b.i)]
\[
|A'|, |B'| \geq (\theta/2-3\varepsilon/2)\varphi(q).
\]
\item[(b.ii)]
\[
|(A' \cdot B') \cap E_2(N)| \geq |A' \cdot B'|- \varepsilon \varphi(q),
\]
\item[(b.iii)]
If $N \geq q^{4/5+\kappa}$, then
\[
(\mathbf{1}_{A' \cdot B'} \ast \mathbf{1}_{A'})(a) \gg \varphi(q) \implies a \in E_3(N).
\]
\item[(b.iv)] Assume that $N \geq q$ and let $\alpha_1 \in (2\varepsilon, 1]$. Write $L = \lfloor 1/\alpha_1\rfloor$. Then
\[
\sum_{\substack{cp \equiv a \pmod{q} \\ q^{\alpha_1-\varepsilon} < p \leq q^{\alpha_1} \\ c \in A' \cdot B'}} \frac{1}{p} \gg 1 \implies a \in E_3(N).
\]
\item[(b.v)]
For any subgroup $H \leq \mathbb{Z}_q^\times$ of index $Y \leq \varepsilon^{-1/2}$, $A' \cap B'$ contains at least $\frac{\varepsilon}{2} \varphi(q)$ elements from at least $\lceil (\theta_0/2-3\varepsilon^{1/2}\theta_0/\theta)Y\rceil$ cosets of $H$. 
\item[(b.vi)] For any $b \in \mathbb{Z}_q^\times$, we have
\[
|A'\cap B' \cap bH|\geq \left(\frac{\theta}{2} \cdot \frac{|E_1(N)\cap bH|}{N/\log N} - \frac{\varepsilon}{3}\right)\varphi(q).
\]
\end{itemize}
\end{enumerate}
\end{proposition}
Proposition~\ref{prop:A'binter-vii} below adds one more conclusion to part (b) of this proposition.

\begin{proof}[Proof of Proposition~\ref{prop:A'binter}]
We may assume that $\varepsilon>0$ is small enough. Let $A$ be as in Proposition~\ref{prop:transfer}.  Note that $|A| \geq (\theta/2-\varepsilon)\varphi(q)$. We apply Lemma~\ref{le:popularkneser} with $t = \lceil \varepsilon \varphi(q)/10\rceil$ and $u = \lfloor \varepsilon^2 \varphi(q)/1000\rfloor$, and with $A$ and $B$ in the lemma both equalling the set $A$.

If Lemma~\ref{le:popularkneser}(a) holds, then the first part of Proposition~\ref{prop:A'binter}(a) holds by Proposition~\ref{prop:transfer}(ii) and the second part of Proposition~\ref{prop:A'binter}(a) holds by Proposition~\ref{prop:transfer}(iii) and Lemma~\ref{le:convolution}(i) applied with $\{a \in \mathbb{Z}_q^\times \colon (\mathbf{1}_{A} \ast \mathbf{1}_{A})(a) \geq u\}$ in place of $A$ and $A$ in place of $B$.

Let us now show that if Lemma~\ref{le:popularkneser}(b) holds then (b.i)--(b.vi) hold. Now (b.i) follows immediately and (b.ii) holds by~Proposition~\ref{prop:transfer}(ii) since $A', B' \subseteq A$ and
\begin{align}
\label{eq:cA'B'}
c \in A' \cdot B'\implies \mathbf{1}_A \ast \mathbf{1}_A(c) \geq \varepsilon \varphi(q)/10.    
\end{align}
Moreover, we have (b.iii) by Proposition~\ref{prop:transfer}(iii) and the estimate
\begin{align*}
(\mathbf{1}_{A' \cdot B'}*\mathbf{1}_{A'})(a)&=\sum_{c\in A' \cdot B'}\mathbf{1}_{A' \cdot B'}(c)\mathbf{1}_{A'}(ac^{-1})\ll \frac{1}{\varphi(q)}\sum_{c\in A' \cdot B'}(\mathbf{1}_{A}*\mathbf{1}_{A})(c)\mathbf{1}_{A'}(ac^{-1}) \\
&\leq \frac{\mathbf{1}_{A}*\mathbf{1}_{A}*\mathbf{1}_{A}(a)}{\varphi(q)}.    
\end{align*}
Claim (b.iv) follows similarly from Proposition~\ref{prop:transfer}(iv).

Since 
\begin{equation}
\label{eq:Aa'minus}
|A \setminus A'| + |A \setminus B'| \leq \frac{\varepsilon}{10}\varphi(q),
\end{equation}
Claim (b.v) follows from Proposition~\ref{prop:transfer}(v) and (b.vi) follows from Proposition~\ref{prop:transfer}(vi).
\end{proof}

\subsection{Structure of sets with small doubling}
In this subsection we provide a consequence of Kneser's theorem which tells us about the structure of $A$ and $B$ in case $A\cdot B$ is small. The following lemma is quite similar to parts of~\cite[Section 4]{BRS} and~\cite[Lemma 3]{szabo}.
\begin{lemma}
\label{le:structure}
Let $\alpha, \alpha', \beta \in (0, 1]$ be such that $\beta < 2\alpha\leq 2\alpha'$,  and let $A, B \subseteq \mathbb{Z}_q^\times$ with $|A|, |B| \geq \alpha \varphi(q)$. Assume that $A$ and $B$ each meet at least proportion $\alpha'$ of cosets of any subgroup $H_0 \leq \mathbb{Z}_q^\times$ of order $<1/(2\alpha-\beta)$. Then at least one of the following holds.
\begin{enumerate}[(a)]
\item We have 
\[
|A \cdot B| \geq \beta \varphi(q).
\]
\item Write $H \leq \mathbb{Z}_q^\times$ for the stabilizer of $A \cdot B$ and write $Y$ for its index. Then the following hold.
\begin{itemize}
    \item[(b.i)] We have
\[
1 < Y < \frac{1}{2\alpha' - \beta}.
\]
    \item[(b.ii)] If $\alpha' > 1/3$ and $\beta \leq 2/3$, then $Y = 3k + 2$ for some integer $k$ with 
\begin{align}\label{eq:alpha'}
\frac{k+1}{3k+2}\geq \alpha',
\end{align}
and each of $A$ and $B$ intersects exactly $k+1$ cosets of $H$, and $A \cdot B$ is the union of $2k+1$ cosets $\pmod{H}$. 

\item[(b.iii)] If, for some small enough $\varepsilon>0$, we have $\alpha' \geq \frac{3}{8}-\varepsilon$ and $\beta< \frac{11}{16}-2\varepsilon$, then $Y = 3k + 2$ for some integer $k\in \{0,1,2\}$, and each of $A$ and $B$ intersects exactly $k+1$ cosets of $H$, and $A \cdot B$ is the union of $2k+1$ cosets $\pmod{H}$. 
\end{itemize}
\end{enumerate}
\end{lemma}

\begin{proof}
If $Y=1$, then $A \cdot B = \mathbb{Z}_q^\times$ and (a) holds. Assume henceforth that $Y>1$. By Kneser's theorem (Lemma~\ref{le:Kneser})
\[
|A \cdot B| \geq |A \cdot H| + |B \cdot H| - |H| \geq |A| + |B| - |H| \geq \left(2\alpha - \frac{1}{Y}\right) \varphi(q).
\]
We see that either (a) holds or
\[
1 < Y < \frac{1}{2\alpha - \beta}.
\]
By assumption, this implies that each of $A$ and $B$ meets at least proportion $\alpha'$ of cosets of $H$. Hence actually
\begin{equation}
\label{eq:A*Alow}
|A \cdot B| \geq |A \cdot H| + |B \cdot H| - |H| \geq  \left(2\left\lceil\alpha' Y \right\rceil - 1 \right)|H| \geq \left(2\alpha' - \frac{1}{Y}\right) \varphi(q)
\end{equation}
and so either (a) holds or (b.i) holds.

We may suppose from now on that (b.i) holds. Then using the second inequality in~\eqref{eq:A*Alow}, we obtain in this case
\begin{align}\label{eq:AB1}
|A \cdot B| \geq \left(2\lceil \alpha'Y\rceil-1\right)|H| = \frac{2\lceil \alpha'Y\rceil-1}{Y}\varphi(q).
\end{align}

If $\alpha'>1/3$ and $\beta\leq 2/3$, then if $Y$ is not $2 \pmod{3}$, we have
\[
\left\lceil\alpha' Y\right\rceil = \left\lceil\frac{Y}{3} + \left(\alpha'-\frac{1}{3}\right) Y\right\rceil \geq \frac{Y}{3}+\frac{2}{3},
\]
so we obtain by~\eqref{eq:AB1}
\[
|A \cdot B| \geq \left(\frac{2Y}{3} + \frac{4}{3} - 1\right)\frac{\varphi(q)}{Y} > \frac{2}{3} \varphi(q)\geq \beta \varphi(q),
\]
so (a) holds. 

If instead $\alpha' \geq 3/8-\varepsilon$ and $\beta< 11/16-2\varepsilon$ with $\varepsilon>0$ small, then (b.i) gives $Y<16$. Moreover, by a case check we see that
\[
\frac{2\lceil \alpha'Y\rceil-1}{Y}\geq \frac{11}{16}
\]
for all $Y\in \{1,\ldots, 15\}$ that are not of the form $Y=3k+2$ with $k\in \{0,1,2\}$. Hence, if $Y$ is not of this form, then~\eqref{eq:AB1} implies that (a) holds.

We may henceforth assume that $Y = 3k+2$ for some $k \in \mathbb{N}$ if $\alpha'>1/3$ and $\beta\leq 2/3$, and that $Y=3k+2$ with $k\in \{0,1,2\}$ if $\alpha'\geq 3/8-\varepsilon$ and $\beta<11/16-2\varepsilon$. Now $A$ and $B$ must each intersect at least 
\[
\lceil \alpha' Y \rceil \geq \left\lceil \frac{1}{3}(3k+2) \right\rceil \geq k+1
\]
cosets of $H$. If $A$ intersects $\geq k+2$ cosets of $H$, then $|A\cdot H|\geq (k+2)|H|$, so by the first inequality of~\eqref{eq:A*Alow} we have
\begin{align*}
|A\cdot B|\geq \left(k+2 + k+1 - 1\right)|H|= (2k+2)\frac{\varphi(q)}{3k+2},
\end{align*}
and now (a) holds since
$(2k+2)/(3k+2)>2/3$ for all $k\in \mathbb{N}$ and $(2k+2)/(3k+2)>11/16$ for $k\in \{0,1,2\}$. Using symmetry, we may assume that each of $A$ and $B$ intersects exactly $k+1$ cosets. 

Now since 
\[
|A\cdot B|\geq |A\cdot H|+|B\cdot H|-|H| = (2k+1)|H|,
\]
we see that $A\cdot B$ must intersect (and hence consist of) at least $2k+1$ cosets of $H$. If $A\cdot B$ is the union of at least $2k+2$ cosets of $H$, then we have
\begin{align*}
|A\cdot B|\geq (2k+2)|H|\geq 2|A| \geq 2\alpha \varphi(q) > \beta \varphi(q) 
\end{align*}
and (a) holds.

Hence if (a) does not hold, then $A$ and $B$ must each be contained in $k+1$ cosets of $H$ and $A\cdot B$ must be the union of $2k+1$ cosets of $H$. Finally, we have~\eqref{eq:alpha'} by the assumption that $A$ meets at least proposition $\alpha'$ of cosets of an index $Y$ subgroup. 
\end{proof}

We shall also use the following which is~\cite[Lemma 4]{szabo}.
\begin{lemma}
\label{le:complstructure}
Let $G$ be an abelian group of order $3k+2$ for some integer $k \geq 0$. Let $B \subseteq G$ with $|B| = k+1$. Assume that $B \cdot B = G \setminus B$ and that the stabilizer of $B \cdot B$ is the identity. Then $G$ is cyclic and for some $g\in G$ which generates $G$ we have $B = \{g^{k+1}, \dotsc, g^{2k+1}\}$.
\end{lemma}

\section{Primes in cosets of indices \texorpdfstring{$5$}{5} and \texorpdfstring{$8$}{8}}
\label{se:leastPrimes}

In order to deal with the case that Lemma~\ref{le:structure}(b) holds, we need to have information on the distribution of primes in cosets of $\mathbb{Z}_q^{\times}$ of index $3k+2$ for small $k$ ($k\in \{0,1,2\}$ suffices for us). We can obtain lower bounds for the logarithmic density of primes in certain collections of such cosets.

\begin{lemma}[Positive logarithmic density of primes in some unions of cosets]
\label{le:k=12}
Let $q\in \mathbb{N}$ be large, let $k \in \{1, 2\}$ and let $H$ be a subgroup of index $3k+2$ such that $\mathbb{Z}_q^\times/H$ is cyclic and generated by $gH$ with $g \in \mathbb{Z}_q^\times$. Then
\begin{equation}
\label{eq:k=12claim}
\sum_{\substack{p \leq q \\ p \not \in \bigcup_{j=1}^{k+1} g^{k+j} H}} \frac{\log p}{p} \geq \frac{1}{1000} \log q.
\end{equation}
\end{lemma}

In the proof of this lemma we use the following lemma which is~\cite[Proposition 6]{szabo} and is based on~\cite[Lemmas 5.2--5.3]{Heath-Brown2}. Also Szab\'o~\cite{szabo} used it in a similar context. 
\begin{lemma}
\label{le:realpart}
Let $q \in \mathbb{N}$ and $\alpha \in \mathbb{R}_{>0}$. Let $f \colon \mathbb{R}_{\geq 0} \to \mathbb{R}$ be defined as
\[
f(t) =
\begin{cases}
\alpha - t & \text{if $0 \leq t \leq \alpha$;} \\
0 & \text{if $t > \alpha$.}
\end{cases}
\]
\begin{enumerate}[(i)]
\item If $\chi$ is a non-principal character $\pmod{q}$ of bounded order, then
\[
\Re \sum_p \frac{\chi(p) \log p}{p} f\left(\frac{\log p}{\log q}\right) \leq \left(\frac{\alpha}{8} + o(1)\right) \log q.
\]
\item We have
\[
\sum_p \frac{\chi_0(p) \log p}{p} f\left(\frac{\log p}{\log q}\right) = \left(\frac{\alpha^2}{2} + o(1)\right) \log q.
\]
\end{enumerate}
\end{lemma}

\begin{proof}[Proof of Lemma~\ref{le:k=12}] Let $J = 2$ if $k=1$ and $J=4$ if $k=2$. Let $\alpha=1$ in Lemma~\ref{le:realpart}. For $j = 0, \dotsc, J$, we define
\[
\beta_j := \frac{\sum_{p \in g^j H \cup g^{-j}H} \frac{\chi_0(p) \log p}{p} f\left(\frac{\log p}{\log q}\right)}{\sum_{p} \frac{\chi_0(p) \log p}{p} f\left(\frac{\log p}{\log q}\right)}.
\]
Trivially
\begin{equation}
\label{eq:beta_jinfo}
\beta_j \in [0, 1] \text{   for each $j = 0, \dotsc, J$ and   } \sum_{j=0}^J \beta_j = 1.
\end{equation}
Let us first consider the case $k=1$, so that $J = 2$. If~\eqref{eq:k=12claim} does not hold, then 
\[
\sum_{\substack{p \leq q \\ p \not \in g^2 H \cup g^3 H}} \frac{\log p}{p} < \frac{1}{1000} \log q.
\]
Hence by definition of $\beta_j$ and Lemma~\ref{le:realpart}(ii)  we have $\beta_0 + \beta_1 \leq (2+o(1)) \cdot \frac{1}{1000}$, and consequently~\eqref{eq:beta_jinfo} implies that $\beta_2 \geq \frac{499}{500}-o(1)$. Let $\chi$ be the character $\pmod{q}$ for which $\chi(g) = e(2/5)$ and $\chi(h) = 1$ for every $h \in H$. Then, since $\cos(0) > \cos(2\pi/5) > 0 > \cos(4\pi/5)$, we have
\begin{align*}
&\Re \sum_p \frac{\chi(p) \log p}{p} f\left(\frac{\log p}{\log q}\right) \\
&= \left(\beta_0 + \beta_1 \cos\left(\frac{4\pi}{5}\right) + \beta_2 \cos\left(\frac{2\pi}{5}\right)\right) \sum_p \frac{\chi_0(p) \log p}{p} f\left(\frac{\log p}{\log q}\right) \\
& \geq \left(\frac{1}{500} \cos\left(\frac{4\pi}{5}\right) + \frac{499}{500} \cos\left(\frac{2\pi}{5}\right) -o(1) \right) \sum_p \frac{\chi_0(p) \log p}{p} f\left(\frac{\log p}{\log q}\right).
\end{align*}
Applying Lemma~\ref{le:realpart}(i) to the left-hand side and Lemma~\ref{le:realpart}(ii) to the right-hand side, we obtain 
\[
\frac{1}{8} \geq \left(\frac{499}{500} \cos\left(\frac{2\pi}{5}\right) + \frac{1}{500} \cos\left(\frac{4\pi}{5}\right)\right) \frac{1}{2}-o(1)=0.153\ldots+o(1).
\]
But this is not true, and we have obtained a contradiction.

Let us now turn to the case $k=2$. If~\eqref{eq:k=12claim} does not hold, then
\[
\sum_{\substack{p \leq q \\ p \not \in g^3 H \cup g^4 H \cup g^5 H}} \frac{\log p}{p} < \frac{1}{1000} \log q.
\]
Hence, by definition of $\beta_j$ and Lemma~\ref{le:realpart}(ii), 
\begin{equation}
\label{eq:b0-2}
\beta_0 + \beta_1 + \beta_2 \leq 2 \cdot \frac{1}{1000} + o(1).
\end{equation}
Let now $\chi$ be the character $\pmod{q}$ for which $\chi(g) = e(1/8)$ and $\chi(h) = 1$ for every $h \in H$. By the definitions of $\beta_j$ and Lemma~\ref{le:realpart}(ii) we obtain
\begin{align*}
&\Re \sum_p \frac{\chi(p)^2 \log p}{p} f\left(\frac{\log p}{\log q}\right) = \left(\beta_0 + \beta_4 - \beta_2\right) \sum_p \frac{\chi_0(p) \log p}{p} f\left(\frac{\log p}{\log q}\right) \\
& \geq \left(\beta_4-\frac{1}{500}+o(1)\right) \sum_p \frac{\chi_0(p) \log p}{p} f\left(\frac{\log p}{\log q}\right) = \left(\frac{\beta_4}{2}-\frac{1}{1000}+o(1)\right) \log q.
\end{align*}
Applying Lemma~\ref{le:realpart}(i) to the left-hand side, we obtain $\beta_4 \leq \frac{1}{4}+\frac{1}{500}+o(1)$, and consequently by~\eqref{eq:beta_jinfo} and~\eqref{eq:b0-2} we must have $\beta_3 \geq \frac{3}{4}-\frac{1}{250}+o(1)$.

On the other hand, using the definition of $\beta_j$ and our bounds for $\beta_j$, we obtain
\begin{align*}
&\Re \sum_p \frac{\chi(p)^3 \log p}{p} f\left(\frac{\log p}{\log q}\right) = \left(\beta_0 - \beta_1 \frac{1}{\sqrt{2}} + \beta_3 \frac{1}{\sqrt{2}} - \beta_4\right) \sum_p \frac{\chi_0(p) \log p}{p} f\left(\frac{\log p}{\log q}\right) \\
& \geq \left(-\frac{1}{500} \frac{1}{\sqrt{2}} + \left(\frac{3}{4}-\frac{1}{250}+o(1)\right) \frac{1}{\sqrt{2}} - \frac{1}{4}-\frac{1}{500}\right) \sum_p \frac{\chi_0(p) \log p}{p} f\left(\frac{\log p}{\log q}\right).
\end{align*}
Applying Lemma~\ref{le:realpart}(i) to the left-hand side and Lemma~\ref{le:realpart}(ii) to the right-hand side, we obtain 
\[
\frac{1}{8} \geq \frac{93}{250\sqrt{2}}-\frac{1}{8}-\frac{1}{1000}+o(1)=0.137\ldots+o(1),
\]
which is a contradiction.
\end{proof}

\begin{remark}
We have not tried to optimize the lower bound in Lemma~\ref{le:k=12}. In our cases $k \in \{1, 2\}$, it would be quite easy to do such an optimization by hand. For larger $k$ one could use linear programming (which has been previously used in a somewhat similar context in~\cite[Proof of Theorem 2]{MatoHecke}) to take full advantage of the information provided by Lemma~\ref{le:realpart}. Even this is not sufficient to prove a similar result in case $k = 3$, but as it turns out we will only need the cases $k\in \{1,2\}$ in the proofs of our theorems. 
\end{remark}

\section{Size of \texorpdfstring{$E_3(N)$}{E3(N)}}
\begin{proof}[Proof of Theorem~\ref{th:E3q}]
Let $\varepsilon>0$ be small. Let $N = q$ if $q$ is cube-free and $N = q^{1+100\varepsilon}$ otherwise. Adjusting $\varepsilon$, it suffices to prove that $E_3(N) = \mathbb{Z}_q^\times$. Recall the definitions of $\theta$ and $\theta_0$ from~\eqref{eq:thetaDef} and~\eqref{eq:theta0Def}. We have
\[
\theta \geq 1-\varepsilon-\frac{\log q}{3\log N} \geq \frac{2}{3} +10\varepsilon
\]
and
\[
\theta_0 =  1-\varepsilon-\frac{\log q}{4\log N} \geq  \frac{3}{4}-\varepsilon.
\]
Let us apply Proposition~\ref{prop:A'binter}. If Proposition~\ref{prop:A'binter}(a) holds, then we immediately get $E_3(N) = \mathbb{Z}_q^\times$.

Hence we can assume that Proposition~\ref{prop:A'binter}(b) holds. Let $A'$ and $B'$ be as there, in particular 
\begin{equation}
\label{eq:A'B'lowerE3}
|A'|, |B'| \geq \left(\frac{\theta}{2}-\frac{3}{2}\varepsilon\right) \varphi(q) \geq \left(\frac{1}{3} + 3\varepsilon\right)\varphi(q).
\end{equation}

Write $H$ for the stabilizer of $A' \cdot B'$. We apply Lemma~\ref{le:structure} with $A = A', B = B',\alpha=1/3+3\varepsilon$, $\alpha'=3/8-10\varepsilon^{1/2}$,  $\beta=2/3-5\varepsilon/2$ (which is an admissible choice of parameters by~\eqref{eq:A'B'lowerE3} and Proposition~\ref{prop:A'binter}(b.v)). In case Lemma~\ref{le:structure}(a) holds, we have $E_3(N) = \mathbb{Z}_q^\times$ by Proposition~\ref{prop:A'binter}(b.iii),~\eqref{eq:A'B'lowerE3}, and Lemma~\ref{le:convolution}(i). 

Hence we can assume that Lemma~\ref{le:structure}(b) holds. Then by Lemma~\ref{le:structure}(b.i)--(b.ii) $A'$ and $B'$ are contained in $k+1$ cosets of a subgroup $H$ of index $Y = 3k+2$ with
\[
1 < Y <\frac{1}{3/4-2/3-30\varepsilon^{1/2}} < 13,
\] 
so that $k \in \{0, 1, 2, 3\}$. Furthermore $A' \cdot B'$ is the union of $2k+1$ cosets of $H$. If $k = 3$ (so that $Y = 11$) then~\eqref{eq:alpha'} fails, so this case cannot actually occur.

Hence we can assume that $k \in \{0, 1, 2\}$. By Proposition~\ref{prop:A'binter}(b.v) we know that $A'$ and $B'$ are contained in the same $k+1$ cosets of $H$. Let now $a_1, \dotsc, a_{k+1}, b_1, \dotsc, b_{2k+1}$ be such that
\[
A', B' \subseteq \bigcup_{j = 1}^{k+1} a_j H
\]
and
\begin{equation}
\label{eq:A'B'union}
A' \cdot B' = \bigcup_{j = 1}^{2k+1} b_j H = \left(\bigcup_{j = 1}^{k+1} a_j H\right)^2.
\end{equation}
By Proposition~\ref{prop:A'binter}(b.v) the set $A'$ contains at least $\frac{\varepsilon}{2} \varphi(q)$ elements from each $a_j H$ with $j \in \{1, \dotsc, k+1\}$. Combining this and~\eqref{eq:A'B'union} with Proposition~\ref{prop:A'binter}(b.iii) and Lemma~\ref{le:convolution}(ii), we see that
\begin{equation}
\label{eq:ajbjcup}
\left(\bigcup_{j = 1}^{k+1} a_j H\right) \cdot \left(\bigcup_{j = 1}^{2k+1} b_j H\right) \subseteq E_3(N).
\end{equation}
The left-hand side does not yet necessarily cover all of $\mathbb{Z}_q^\times$. To proceed, we split into two cases.
\subsection{Case 1: \texorpdfstring{$k \in \{1, 2\}$}{k in {1,2}} or \texorpdfstring{$k = 0$}{k=0} and \texorpdfstring{$a_1 \in H$}{a1 in H}}
We shall in a moment show that in this case there exists $\alpha_1 \in (2\varepsilon, 1]$ such that
\begin{equation}
\label{eq:alclaim}
\sum_{\substack{q^{\alpha_1-\varepsilon} < p \leq q^{\alpha_1} \\ p \not \in \cup_{j = 1}^{k+1} a_j H}} \frac{\log p}{p} \gg \log q.
\end{equation}
Let us first show how~\eqref{eq:alclaim} implies the claim that $E_3(N) = \mathbb{Z}_q^\times$. By~\eqref{eq:alclaim} there exists $a_0 \in \mathbb{Z}_q^\times$ such that $a_0 H \not \in \{a_1H, \dotsc, a_{k+1}H\}$ and 
\[
\sum_{\substack{q^{\alpha_1-\varepsilon} < p \leq q^{\alpha_1} \\ p \in a_0 H}} \frac{\log p}{p} \gg \log q.
\]
Then, for each $a \in a_0 H \cdot \bigcup_{j = 1}^{2k+1} b_j H$ we have by~\eqref{eq:A'B'union}
\[
\sum_{\substack{cp \equiv a \pmod{q} \\ q^{\alpha_1-\varepsilon} < p \leq q^{\alpha_1} \\ c \in A' \cdot B'}} \frac{1}{p} \geq \sum_{\substack{q^{\alpha_1-\varepsilon} < p \leq q^{\alpha_1} \\ p \in a_0 H}} \frac{1}{p} \gg 1.
\]
Hence by Proposition~\ref{prop:A'binter}(b.iv) we obtain that 
\[
a_0 H \cdot \bigcup_{j = 1}^{2k+1} b_j H \subseteq E_3(N).
\]
Combining this with~\eqref{eq:ajbjcup} and Lemma~\ref{le:convolution}(i) we see that $E_3(N) = \mathbb{Z}_q^\times$.

Hence the remaining task is to prove~\eqref{eq:alclaim}. Consider first the case that $\cup_{j = 1}^{k+1} a_j H$ and $\cup_{j = 1}^{2k+1} b_j H$ are not complements of each other. In this case there exists $b_0 \in \mathbb{Z}_q^\times$ such that $b_0H$ intersects neither of these unions. By Lemma~\ref{le:P2} we know that $b_0H \cap [1, q]$ contains $\gg q/\log q$ integers who have at most two prime factors and whose prime factors are $\geq q^{1/3}$. Since $b_0 \not \in \cup_{j = 1}^{k+1} a_j H$ by Proposition~\ref{prop:A'binter}(b.vi), there must be $\gg q/\log q$ products of exactly two primes $\geq q^{1/3}$ in $b_0H \cap [1, q]$. But since
\[
b_0 \not \in \bigcup_{j = 1}^{2k+1} b_j H = \left(\bigcup_{j = 1}^{k+1} a_j H\right)^2,
\]
we see that both prime factors cannot be from $\cup_{j = 1}^{k+1} a_j H$. Consequently
\[
\frac{q}{\log q} \ll \sum_{\substack{q^{1/3} \leq p \leq q^{2/3} \\ p \not \in \cup_{j = 1}^{k+1} a_j H}} \sum_{q^{1/3} \leq p' \leq q/p} 1 \ll \sum_{\substack{q^{1/3} \leq p \leq q^{2/3} \\ p \not \in \cup_{j = 1}^{k+1} a_j H}} \frac{q}{p\log q/p}
\]
and~\eqref{eq:alclaim} follows.

Note that if $k=0$ and $a_1 \in H$, then definitely $a_1 H = H$ and $b_1 H = a_1^2 H = H$ are not complements of each other, and so in the rest of the proof of case 1 we can assume that $k \in \{1, 2\}$ and $\cup_{j = 1}^{k+1} a_j H$ and $\cup_{j = 1}^{2k+1} b_j H$ are complements of each other. But then~\eqref{eq:alclaim} follows from Lemmas~\ref{le:complstructure} and~\ref{le:k=12}.

\subsection{Case 2: \texorpdfstring{$k=0$}{k=0} and \texorpdfstring{$a_1 \not \in H$}{a1 not in H}}
\label{ssec:NQRcase}
This is the most difficult case; due to the possible existence of exceptional characters, there may exist a quadratic character $\psi$ such that $\psi(p) = -1$ for almost all primes, and thus we might indeed have $p \not \in H$ for almost all primes $p$. We need two auxiliary results before we can finish the proof of Theorem~\ref{th:E3q} in this case. In this subsection we state these results and finish the proof of Theorem~\ref{th:E3q} asssuming them. We postpone the proofs of the auxiliary results to Section~\ref{se:QR}.

First, while we cannot show that the primes with $\psi(p) = 1$ have positive logarithmic density, we can show the following.

\begin{lemma}[Lower-bounding the number of primes with $\psi(p) = 1$]\label{le:QR}
Let $q \in \mathbb{N}$ be sufficiently large and let $\psi\pmod{q}$ be a real character. Then there exists a positive constant $c_0$ such that either
\begin{align}
\label{eq:LotofQR}
\sum_{\substack{q^{1/4} < p \leq q^{5/7} \\ \psi(p) = 1}} \frac{1}{p} \geq \frac{1}{10}
\end{align}
or there exists $y \in [q^{5/7}, q]$ such that
\begin{align}
\label{eq:SomeQR}
\sum_{\substack{p \leq y \\ \psi(p) = 1}} 1 \geq c_0 y L(1, \psi)\frac{\varphi(q)}{q} \prod_{\substack{2<p \leq q \\ \psi(p) = 1}} \left(1-\frac{2}{p}\right).
\end{align}
\end{lemma}

This lemma has some similarities with the work of Dunn, Kerr, Shparlinski and Zaharescu~\cite[Theorem 1.1]{dunn} (see also work of Benli~\cite[Theorem 4]{benli}, but there is an inaccuracy\footnote{The author of~\cite{benli} is in the process of correcting this.} in the proof: In the proof of~\cite[Lemma 9]{benli} the $k$-dimensional fundamental lemma of the sieve is currently applied with level $z^s$ larger than the upper bound for $e$ in \cite[Proposition 7]{benli} --- however the methods in~\cite{benli} definitely do yield a result like~\cite[Theorem 4]{benli}, possibly with a slightly different lower bound).

\begin{remark}
The bound~\eqref{eq:LotofQR} is clearly optimal up to a constant factor, and similarly the lower bound in~\eqref{eq:SomeQR} is optimal up to a constant factor (see Lemma~\ref{le:QRupper} below).  
\end{remark} 

The second important auxiliary result is the following addition to Proposition~\ref{prop:A'binter}.
\begin{proposition}
\label{prop:A'binter-vii}
In the set-up of Proposition~\ref{prop:A'binter}, in case Proposition~\ref{prop:A'binter}(b) holds, the sets $A', B'$ satisfy in addition to (b.i)--(b.vi) also the following: 

Assume that $N \geq q$ and $M \in [q^{2/3+3\varepsilon}, q]$. Then
\[
\sum_{\substack{c p \equiv a \pmod{q} \\ p \sim M, \psi(p) = 1 \\ c \in A' \cdot B'}} 1 \gg  M L(1, \psi) \frac{\varphi(q)}{q}\prod_{\substack{2<p \leq q}} \left(1-\frac{2}{p}\right) \implies a \in E_3(N).
\]
\end{proposition}

Now we finish the proof of Theorem~\ref{th:E3q} in the remaining case $k=0$ and $a_1 \not \in H$, assuming Lemma~\ref{le:QR} and Proposition~\ref{prop:A'binter-vii}. In this case $A', B' \subseteq a_1H$ and $A' \cdot B' = a_1^2H = H$. 

By~\eqref{eq:ajbjcup} it suffices to show that $a \in E_3(N)$ for every $a \in H$. Since $A' \cdot B' = H$, by Proposition~\ref{prop:A'binter-vii} this follows if, for some $M \in [q^{2/3+3\varepsilon}, q]$, we have
\[
\sum_{\substack{p \sim M \\ \psi(p) = 1}} 1 \gg M L(1, \psi)\frac{\varphi(q)}{q} \prod_{\substack{2<p < q \\ \psi(p) = 1}} \left(1-\frac{2}{p}\right).
\]
But this either follows from Lemma~\ref{le:QR} or~\eqref{eq:LotofQR} holds. But if~\eqref{eq:LotofQR} holds, then \eqref{eq:alclaim} holds and the argument in case 1 goes through.
\end{proof}

\section{Proofs of Lemma~\ref{le:QR} and Proposition~\ref{prop:A'binter-vii}}
\label{se:QR}
In the proofs of Lemma~\ref{le:QR} and Proposition~\ref{prop:A'binter-vii} we will apply a sieve to the sequence $(1 \ast \psi)(n)$. In the course of doing this, we shall need the fundamental lemma of the sieve (see e.g.~\cite[Lemma 6.8]{friedlander}):
\begin{lemma}[Fundamental lemma of the sieve]
\label{le:fundsieve}
Let $\kappa \geq 1$ be fixed. Let $z\geq 2$ and let $D = z^s$ with $s\geq 9\kappa + 1$. There exist coefficients $\lambda_d^\pm$ such that the following hold.
\begin{enumerate}[(i)]
\item $|\lambda_d^\pm| \leq 1$ for every $d \in \mathbb{N}$ and $\lambda_d^\pm$ are supported on $\{d \leq D:\,\, d \mid P(z)\}$.
\item For every $n \in \mathbb{N}$,
\[
\sum_{d \mid n} \lambda^-_d \leq \mathbf{1}_{(n, P(z)) = 1} \leq \sum_{d \mid n} \lambda^+_d. 
\]
\item If $h \colon \mathbb{N} \to [0, 1)$ is a multiplicative function such that, for some $K \geq 1$, one has
\[
\prod_{w_1 \leq p < z_1} (1-h(p))^{-1} \leq K \left(\frac{\log z_1}{\log w_1}\right)^\kappa
\]
for any $z_1 \geq w_1 \geq 2$, then
\begin{align*}
\sum_{d \mid P(z)} \lambda_d^+ h(d) &\leq (1+e^{9\kappa-s} K^{10}) \prod_{p < z} (1-h(p)), \\
\sum_{d \mid P(z)} \lambda_d^- h(d) &\geq (1-e^{9\kappa-s} K^{10}) \prod_{p < z} (1-h(p)).
\end{align*}
\end{enumerate}
\end{lemma}
Before moving on, let us collect some arithmetic information concerning $(1 \ast \psi)(n)$.
\begin{lemma}
\label{le:1astchi}
Let $q \in \mathbb{N}$ and let $\psi$ be a quadratic character $\pmod{q}$. For each $\varepsilon > 0$, there exists a constant $\eta = \eta(\varepsilon)$ such that, for any $x \geq q^{1/4+\varepsilon}$,
\[
\sum_{n \leq x} (1 \ast \psi)(n) = L(1, \psi) x + O_\varepsilon(x^{1-\eta}).
\]
\end{lemma}
\begin{proof} This follows easily from the Dirichlet hyperbola method and the Burgess bound, see~\cite[Proposition 3.1]{PollackFirstSeveral}.
\end{proof}

\begin{lemma}
\label{le:1astchi_d} Let $q \in \mathbb{N}$ and let $\psi$ be a quadratic character $\pmod{q}$. For each $\varepsilon > 0$, there exists a constant $\eta = \eta(\varepsilon)$ such that whenever $d \in \mathbb{N}$ and $x/d \geq q^{1/4+\varepsilon}$, we have
\begin{align}
\label{eq:1astchi_d}
\sum_{\substack{n \leq x \\ d \mid n}} (1 \ast \psi)(n) = h(d) L(1, \psi) x + O_\varepsilon(\tau_3(d)\tau(d)(x/d)^{1-\eta}),
\end{align}
where $h$ is a multiplicative function such that, for every $p \in \mathbb{P}$,
\begin{align}
\label{eq:hdef}
h(p) = \frac{1+\psi(p)}{p}-\frac{\psi(p)}{p^2} \quad \text{and} \quad |h(p^k)| \leq \frac{2k}{p^k} \, \text{ for any $k \in \mathbb{N}$.}
\end{align}
\end{lemma}
\begin{proof}
This follows easily from Lemma~\ref{le:1astchi} (and is essentially a special case of~\cite[Proposition 7]{benli}). Let
\[
S_d := \sum_{\substack{n \leq x \\ d \mid n}} (1 \ast \psi)(n) = \sum_{\substack{m_1 m_2 \leq x \\ d \mid m_1 m_2}} \psi(m_2). 
\]
Writing $d = d_1 d_2$ with $d_1 = (d, m_1)$,  we see that
\[
S_d = \sum_{d = d_1 d_2} \sum_{\substack{m_1 \\ d_1 \mid m_1 \\ (d_2, m_1/d_1) = 1}} \sum_{\substack{m_2 \leq x/m_1 \\ d_2 \mid m_2}} \psi(m_2). 
\]

We remove the condition $(d_2, m_1/d_1) = 1$ by M\"obius inversion, obtaining 
\[
S_d = \sum_{d = d_1 d_2} \sum_{d_2 = fg} \mu(f) \sum_{\substack{m_1 \\ d_1 \mid m_1 \\ f \mid m_1/d_1}} \sum_{\substack{m_2 \leq x/m_1 \\ d_2 \mid m_2}} \psi(m_2). 
\]
Making the change of variables $m_1 \to d_1 f m_1$ and $m_2 \to m_2 d_2 = m_2 fg$, we obtain
\[
S_d = \sum_{d = d_1 f g} \mu(f)\psi(fg) \sum_{m_1 m_2 \leq x/(d_1 f^2 g)} \psi(m_2) = \sum_{d = d_1 f g} \mu(f)\psi(fg) \sum_{m \leq \frac{x}{df}} (1 \ast \psi)(m).
\]
Now $f \geq (x/d)^{\varepsilon/2}$ contribute
\[
\ll \tau_3(d) \sum_{\substack{f\mid d\\f\geq (x/d)^{\varepsilon/2}}}\frac{x}{df} \log \frac{x}{d} \ll \tau_3(d)\tau(d)\left(\frac{x}{d}\right)^{1-\varepsilon/2} \log \frac{x}{d},
\]
whereas by Lemma~\ref{le:1astchi} the contribution  of $f < (x/d)^{\varepsilon/2}$ is, for some $\eta > 0$,
\begin{align*}
&\sum_{\substack{d = d_1 f g \\ f < (x/d)^{\varepsilon/2}}} \mu(f)\psi(fg) \left(L(1, \psi) \frac{x}{df} + O\left(\left(\frac{x}{df}\right)^{1-\eta}\right)\right)\\
&= x L(1, \psi) \sum_{\substack{d = d_1 f g}} \frac{\mu(f)\psi(fg)}{d_1 f^2 g} + O\left(\tau_3(d)\left(\frac{x}{d}\right)^{1-\eta/2}\right).
\end{align*}
Writing out the Euler product we obtain~\eqref{eq:1astchi_d} with $h$ being the multiplicative function such that, for any $p \in \mathbb{P}$ and $k \in \mathbb{N}$,
\begin{align*}
h(p^k) &= \sum_{p^k = d_1 g} \frac{\psi(g)}{d_1 g} - \sum_{p^k = d_1 p g} \frac{\psi(pg)}{d_1 p^2 g} = \sum_{j=0}^k \frac{\psi(p^j)}{p^k} - \sum_{j = 0}^{k-1} \frac{\psi(p^{j+1})}{p^{k+1}} \\
&= \frac{\psi(p^k)}{p^k} + \frac{1}{p^k} \sum_{j = 0}^{k-1} \psi(p^j) \left(1-\frac{\psi(p)}{p}\right)
\end{align*}
which clearly satisfies~\eqref{eq:hdef}.
\end{proof}

Before turning to the proof of Lemma~\ref{le:QR}, let us quickly combine Lemmas~\ref{le:fundsieve} and~\ref{le:1astchi_d} to obtain an upper bound corresponding to~\eqref{eq:SomeQR}.

\begin{lemma}[Upper-bounding the number of primes with $\psi(p) = 1$]\label{le:QRupper}
Let $\varepsilon > 0$ be fixed, let $q \in \mathbb{N}$ be sufficiently large and let $\psi\pmod{q}$ be a real character. Then there exists a positive constant $C_{\varepsilon}$ such that, for any $y \in [q^{1/4+\varepsilon}, q]$ we have
\begin{align*}
\sum_{\substack{p \leq y \\ \psi(p) = 1}} 1 \leq C_{\varepsilon} y L(1, \psi)\frac{\varphi(q)}{q} \prod_{\substack{2<p \leq q \\ \psi(p) = 1}} \left(1-\frac{2}{p}\right).
\end{align*}
\end{lemma}

\begin{proof}
We can clearly assume that $\varepsilon$ is small. Let $\lambda_d^+$ be as in Lemma~\ref{le:fundsieve} with $\kappa = 2$, sifting parameter $z = q^{\varepsilon^2}$, and level $D = q^{\varepsilon/2}$ (so that $s = 1/(2\varepsilon)$).

Now by Lemma~\ref{le:fundsieve}(ii) and Lemma~\ref{le:1astchi_d}, we have, for some $\eta > 0$,
\begin{align*}
\sum_{\substack{p \leq y \\ \psi(p) = 1}} 1 &\leq \sum_{p \leq y} (1 \ast \psi)(p) \leq \sum_{n \leq y} (1 \ast \psi)(n) \sum_{\substack{e \mid n}} \lambda_e^+ \leq  \sum_{\substack{e \mid P(q^{\varepsilon^2})}} \lambda_e^+ \sum_{\substack{n \leq y \\ e \mid n}} (1 \ast \psi)(n) \\
&= y L(1, \psi) \sum_{e \mid P(q^{\varepsilon^2})} \lambda_e^+ h(e) + O(y^{1-\eta}),
\end{align*}
where $h$ is as in Lemma~\ref{le:1astchi_d}. Now the claim follows from the fact that $L(1,\psi)\gg_{\delta} q^{-\delta}$ by Siegel's theorem and Lemma~\ref{le:fundsieve}(iii) since
\begin{align}
\label{eq:hprod}
\prod_{p < q^{\varepsilon^2}} (1-h(p)) \asymp \prod_{p \leq q} (1-h(p)) \asymp \prod_{p \mid q} \left(1-\frac{1}{p}\right) \prod_{\substack{2<p \leq q \\ \psi(p) = 1}} \left(1-\frac{2}{p}\right) = \frac{\varphi(q)}{q}  \prod_{\substack{2<p \leq q \\ \psi(p) = 1}} \left(1-\frac{2}{p}\right).
\end{align}
\end{proof}

\begin{proof}[Proof of Lemma~\ref{le:QR}]
Let $w = q^{\varepsilon^4}$, where $\varepsilon > 0$ is small but fixed. Let $a_n = (1\ast \psi)(n) \mathbf{1}_{(n, P(w)) = 1}$. Note that $a_n \geq 0$ for all $n$. Let $\lambda_d^\pm$ be as in Lemma~\ref{le:fundsieve} with $\kappa = 2$, sifting parameter $z = w$, and level $D = q^{\varepsilon^2}$ (so that $s = \varepsilon^{-2}$).

We use Chebychev's method like Dunn, Kerr, Shparlinski and Zaharescu do in the proof of~\cite[Theorem 1.1]{dunn}. Now by Lemma~\ref{le:fundsieve}(ii), Lemma~\ref{le:1astchi_d}, and partial summation, we have, for some $\eta > 0$,
\begin{align}
\begin{aligned}
\label{eq:Cheby}
\sum_{d \leq q} \Lambda(d) \sum_{\substack{n \leq q \\ d \mid n}} a_n &= \sum_{n \leq q} a_n \sum_{d \mid n} \Lambda(d) = \sum_{n \leq q} a_n \log n \geq \sum_{e \mid P(w)} \lambda_e^- \sum_{\substack{n \leq q \\ e \mid n}} (1 \ast \psi)(n)(\log n) \\
&= q (\log q) L(1, \psi) \sum_{e \mid P(w)} \lambda_e^- h(e) + O(q^{1-\eta}).
\end{aligned}
\end{align}

By Lemma~\ref{le:fundsieve}(ii) and Lemma~\ref{le:1astchi_d}, we have, for some $\eta > 0$,
\begin{align}
\begin{aligned}
\label{eq:<5/7}
\sum_{\substack{d \leq q^{5/7} \\ (d, P(w)) = 1}} \Lambda(d) \sum_{\substack{n \leq q \\ d \mid n}} a_n &\leq \sum_{\substack{d \leq q^{5/7} \\ (d, P(w)) = 1}} \Lambda(d) \sum_{e \mid P(w)} \lambda_e^+ \sum_{\substack{n \leq q \\ de \mid n}} (1 \ast \psi)(n) \\
&= L(1, \psi) q \sum_{e \mid P(w)} \lambda_e^+ h(e) \cdot \sum_{\substack{d \leq q^{5/7} \\ (d, P(w)) = 1}} \Lambda(d) h(d)  + O(q^{1-\eta}). 
\end{aligned}
\end{align}

By~\eqref{eq:hdef} and Mertens' theorem,
\begin{align}
\begin{aligned}
\label{eq:<1/4}
\sum_{\substack{d \leq q^{1/4} \\ (d, P(w)) = 1}} \Lambda(d) h(d) &\leq \sum_{w \leq p \leq q^{1/4}} \frac{2\log p}{p} + \sum_{\substack{p^k \leq q^{1/4} \\ p \geq w, k \geq 2}} \frac{2k}{p^k} \log p \leq \frac{1}{2} \log q.
\end{aligned}
\end{align}
Furthermore, if~\eqref{eq:LotofQR} does not hold, then since $|h(p)| > 1/p^2$ for $p\nmid q$ implies $\psi(p)=1$, we have
\begin{equation}
\label{eq:1/4-5/7}
\sum_{\substack{q^{1/4} \leq d \leq q^{5/7} \\ (d, P(w)) = 1}} \Lambda(d) h(d) < 2 \cdot \frac{1}{10} \log q^{5/7} +O(q^{-1/10}) = \frac{1}{7} \log q + O(q^{-1/10}).
\end{equation}

Inserting~\eqref{eq:<1/4} and~\eqref{eq:1/4-5/7} into~\eqref{eq:<5/7} and noting that $L(1,\psi)\sum_{e \mid P(w)} \lambda_e^+ h(e) > 0$, we see that, for some $\eta > 0$,
\[
\sum_{\substack{d \leq q^{5/7} \\ (d, P(w)) = 1}} \Lambda(d) \sum_{\substack{n \leq q \\ d \mid n}} a_n \leq \frac{9}{14} q  L(1, \psi) (\log q)  \sum_{e \mid P(w)} \lambda_e^+ h(e) + O(q^{1-\eta}). 
\]
Once $\varepsilon$ is sufficiently small, this together with~\eqref{eq:Cheby} and Lemma~\ref{le:fundsieve}(iii) implies that, for some $\eta > 0$,
\begin{equation}
\label{eq:AfterCheb}
\sum_{q^{5/7} < d \leq q} \Lambda(d) \sum_{\substack{n \leq q \\ d \mid n}} a_n \geq \frac{1}{3}  q (\log q) L(1, \psi) \prod_{p < w} (1-h(p)) + O(q^{1-\eta}).
\end{equation}
By Siegel's bound, the right-hand side is $\gg_{\delta}q^{1-\delta}$ for any $\delta > 0$. Note that higher prime powers give a negligible contribution to the left-hand side, $a_n$ are supported on numbers whose all prime factors $p$ satisfy $\psi(p) \in \{0, 1\}$ and $p \geq w$. Furthermore $a_n = O_\varepsilon(1)$. Hence~\eqref{eq:AfterCheb} implies 
\[
q (\log q) L(1, \psi) \prod_{p < w} (1-h(p)) \ll \sum_{\substack{q^{5/7} < p \leq q \\ \psi(p) \in \{0, 1\}}} \log p \frac{q}{p \log w} + \sum_{\substack{p \leq q \\ \psi(p) \in\{0, 1\}}} \log p
\]
Now the primes with $\psi(p) = 0$ make a negligible contribution and we obtain 
\[
q (\log q) L(1, \psi) \prod_{p < w} (1-h(p))  \ll \sum_{\substack{q^{5/7}/2\leq y<q/2\\y=2^{\ell}}}\frac{1}{y}\sum_{\substack{p\sim y \\q^{5/7}<p\leq q\\ \psi(p) \in  \{0, 1\}}}1+(\log q)\sum_{\substack{p \leq q \\ \psi(p) = 1}}1.
\]
and the claim follows.
\end{proof}

We will need the following variant of the Hal{\'a}sz--Montgomery type mean value theorem (Lemma~\ref{le:Hal-Mon}) which takes into account the possible sparsity of the set of primes $p$ with $\psi(p) = 1$.

\begin{lemma}[Sharp Hal\'asz--Montgomery with weight $1*\psi$] \label{le:Hal-MonL1chi}
Let $\varepsilon > 0$ and $C \geq 1$ be fixed, $q \in \mathbb{N}$, and let $N \in [q^{2/3+3\varepsilon}, q^C]$. Let $\chi_1, \dotsc, \chi_R$ be distinct Dirichlet characters of modulus $q$ and let $\psi$ be a quadratic character of modulus $q$. Then there exists $\eta = \eta(\varepsilon) > 0$ such that, for any complex coefficients $a_n$,
\begin{align*}
&\sum_{j = 1}^R \left|\sum_{\substack{n \leq N \\ (n, P(q^\varepsilon)) = 1}} (1 \ast \psi)(n) a_n \chi_j(n)\right|^2 \\
&\ll_\varepsilon \left(N L(1, \psi) \frac{\varphi(q)}{q} \prod_{\substack{2 < p \leq q \\ \psi(p) = 1}} \left(1-\frac{2}{p}\right) + Nq^{-\eta} R\right) \sum_{\substack{n \leq N \\ (n, P(q^\varepsilon)) = 1}} (1 \ast \psi)(n) |a_n|^2.
\end{align*}
\end{lemma}

\begin{proof}
The proof is somewhat similar to the proof of Lemma~\ref{le:Hal-Mon} although there are some additional complications. By the duality principle (see e.g.~\cite[Section 7.1]{iw-kow}), it suffices to show that, for any complex coefficients $c_j$,
\begin{align}
\label{eq:dualclaimQR}
\sum_{\substack{n \leq N \\ (n, P(q^\varepsilon)) = 1}} (1 \ast \psi)(n) \left|\sum_{j = 1}^R c_j \chi_j(n)\right|^2 \ll \left(N L(1, \psi) \frac{\varphi(q)}{q} \prod_{\substack{2 < p \leq q \\ \psi(p) = 1}} \left(1-\frac{2}{p}\right) + Nq^{-\eta} R\right) \sum_{\substack{j=1}}^R |c_j|^2. 
\end{align}
Let $\lambda_d^+$ be as in Lemma~\ref{le:fundsieve} with $\kappa = 2$, sifting parameter $z = q^{\varepsilon^2}$ and level $D = q^\varepsilon$ (so that $s = \varepsilon^{-1}$). Then the left-hand side of~\eqref{eq:dualclaimQR} is by Lemma~\ref{le:fundsieve}(ii)
\begin{align}\label{eq:1*chi}
&\leq \sum_{d \leq D} \lambda_d^+ \sum_{\substack{n \leq N \\ d \mid n}} (1 \ast \psi)(n)  \left|\sum_{j = 1}^R c_j \chi_j(n)\right|^2 = \sum_{j, k = 1}^R c_j \overline{c_k}  \sum_{d \leq D} \lambda_d^+ \sum_{\substack{n \leq N \\ d \mid n}} \chi_j(n)\overline{\chi_k(n)} (1 \ast \psi)(n).
\end{align}

Note that for a character $\chi\pmod q$ with $\chi \not \in \{\chi_0,\psi\}$, by the hyperbola method and the Burgess bound (Lemma~\ref{le:Burgess}) we have
\begin{align*}
\sum_{\substack{n\leq N\\d\mid n}}\chi(n)(1*\psi)(n)=\sum_{e\leq N^{1/2}}\psi(e)\sum_{\substack{n\leq N\\d\mid n, e\mid n}}\chi(n)+\sum_{e\leq N^{1/2}}\overline{\psi}(e)\sum_{\substack{eN^{1/2}<n\leq N\\d\mid n, e\mid n}}\chi(n)\psi(n)\ll \frac{N}{d}q^{-2\eta}
\end{align*}
for some $\eta=\eta(\varepsilon)>0$, provided that $N^{1/2}/d\geq q^{1/3+\varepsilon/10}$. Using this and the assumption $N\geq q^{2/3+3\varepsilon}$, we see that~\eqref{eq:1*chi} for $\chi_j \overline{\chi_k} \not \in \{\chi_0, \psi\}$ is
\[
\ll R\sum_{j = 1}^R |c_j|^2 Nq^{-\eta}.
\]

For the remaining pairs we note that 
\[
(1 \ast \psi)(n)\psi(n) = (1 \ast \psi)(n) \chi_0(n) = (1 \ast \psi)(n) \mathbf{1}_{(n, q) = 1}
\]
and thus they contribute to~\eqref{eq:1*chi} at most 
\begin{align}\label{eq:c_j2}
\left|\sum_{\substack{j, k = 1 \\ \chi_j \overline{\chi_k} \in \{\chi_0, \psi\}}}^R c_j \overline{c_k}  \sum_{d \leq D} \lambda_d^+  \sum_{\substack{n \leq N \\ d \mid n \\ (n, q) = 1}} (1 \ast \psi)(n)\right| \leq 2\sum_{\substack{j = 1}}^R |c_j|^2 \left|\sum_{\substack{d \leq D \\ (d, q) = 1}} \lambda_d^+ \sum_{\substack{n \leq N \\ d \mid n \\ (n, q) = 1}} (1 \ast \psi)(n)\right|.
\end{align}
By M\"obius inversion, Lemma~\ref{le:1astchi_d} and~\eqref{eq:hdef}, we have, for any $d \leq D$ with $(d,q)=1$ and some $\eta'> 0$, 
\begin{align*}
\sum_{\substack{n \leq N \\ d \mid n \\ (n, q) = 1}} (1 \ast \psi)(n) &=\sum_{e\mid q}\mu(e)\sum_{\substack{n \leq N \\ de\mid n}} (1 \ast \psi)(n) \\
& =\sum_{\substack{e \mid q \\ de \leq N/q^{1/4+\varepsilon}}} \mu(e) \sum_{\substack{n \leq N \\ de\mid n}} (1 \ast \psi)(n) + O\left(\sum_{\substack{e \mid q \\ de > N/q^{1/4+\varepsilon}}}  \sum_{\substack{n \leq N \\ de\mid n}} \tau(n) \right) \\
&=\sum_{\substack{e\mid q \\ de \leq N/q^{1/4+\varepsilon}}}\mu(e)\left( L(1,\psi)h(de)N+O\left((de)^{\eta'/10}\left(\frac{N}{de}\right)^{1-\eta'}\right)\right)+O\left(q^{1/3}\right)\\
&=h(d) N L(1,\psi) \prod_{p\mid q}(1-h(p))+O\left(\left(\frac{N}{d}\right)^{1-\eta'/2}\right).
\end{align*}

Thus, using Lemma~\ref{le:fundsieve}(iii) and a variant of~\eqref{eq:hprod} we can bound the right-hand side of~\eqref{eq:c_j2} as
\begin{align*}
&= 2 \sum_{\substack{j = 1}}^R |c_j|^2 NL(1, \psi)\prod_{p\mid q}(1-h(p))\sum_{\substack{d \leq D \\ (d, q) = 1}} h(d) \lambda_d^++O\left(\sum_{j=1}^R |c_j|^2\sum_{\substack{d\leq D\\(d,q)=1}}\left(\frac{N}{d}\right)^{1-\eta'/2}\right)\\
&\leq 4\sum_{\substack{j = 1}}^R |c_j|^2 NL(1, \psi)\prod_{p\mid q}(1-h(p))\cdot \prod_{\substack{p < z\\(p,q)=1}} (1-h(p))+O\left(\sum_{j=1}^R |c_j|^2N^{1-\eta'/4}\right) \\
&\ll \sum_{\substack{j = 1}}^R |c_j|^2 NL(1, \psi)\frac{\varphi(q)}{q} \prod_{\substack{2 < p \leq q \\ \psi(p) = 1}} \left(1-\frac{2}{p}\right)
\end{align*}
by Siegel's bound $L(1,\psi)\gg_{\delta}q^{-\delta}$. Now the claim follows.
\end{proof}

This allows us to add one more conclusion to Proposition~\ref{prop:transfer}:
\begin{proposition}
\label{prop:transferQR}
In the set-up of Proposition~\ref{prop:transfer}, the set $A$ satisfies in addition to (i)--(vi) also the following: 

Assume that $N \geq q$ and $M \in [q^{2/3+3\varepsilon}, q]$. Then
\[
\sum_{\substack{p a_1 a_2 \equiv a \pmod{q} \\ p \sim M, \, \psi(p) = 1 \\ a_1, a_2 \in A}} 1 \gg  \frac{\varphi(q)}{(\log q)^{1/2-\varepsilon}} M L(1, \psi) \frac{\varphi(q)}{q}\prod_{\substack{2<p \leq q \\ \psi(p) = 1}} \left(1-\frac{2}{p}\right) \implies a \in E_3(N).
\]
\end{proposition}

\begin{proof}
We proceed similarly to the proof of Proposition~\ref{prop:transfer}(iv) but with different choice of $f_0$. We let $f \colon [N]_q \to \mathbb{R}_{\geq 0}$ be as in~\eqref{eq:fDef} and let $f_0 \colon [N]_q \to \mathbb{R}_{\geq 0}$ be defined through
\[
f_0(n) = \mathbf{1}_{n \in \mathbb{P} \cap (M, 2M]}\mathbf{1}_{(n, q) = 1} (1 \ast \psi)(n).
\]
Note that $f_0$ is supported on primes with $\psi(p)=1$ and definitely $a \in E_3(N)$ if $(f \ast f \ast f_0)(n) > 0$ for some $n \equiv a \pmod{q}$. Recall that $f$ satisfies the assumptions in Proposition~\ref{prop:f=g+h} with $\delta = (\log q)^{-1/2+\varepsilon/2}$ and $\eta = o(1)$. Let $g$ then be as in Proposition~\ref{prop:f=g+h}.

Now by orthogonality of characters and Proposition~\ref{prop:f=g+h}(iii) (abusing the notation to consider $f_0$ as a function from $\mathbb{Z}_q^\times$ when necessary)
\begin{align}
\begin{aligned}
\label{eq:f0ff-f0ggQR}
&\frac{1}{|[N]_q|^2} \sum_{n \equiv a \pmod{q}} (f \ast f \ast f_0)(n) - \frac{1}{\varphi(q)^2} (g \ast g \ast f_0)(a) \\
&=\frac{1}{\varphi(q)} \sum_{\chi \pmod{q}} \left(\left(\mathbb{E}_{n \in [N]_q} f(n) \overline{\chi}(n)\right)^2 - \left(\mathbb{E}_{b \in \mathbb{Z}_q^\times} g(b) \overline{\chi}(b)\right)^2\right)  \sum_{n \sim M} f_0(n) \overline{\chi(n)} \chi(a) \\
&=O\bigg(\frac{1}{\varphi(q)} \sum_{\chi \pmod{q}} \left|\mathbb{E}_{n \in [N]_q} f(n) \overline{\chi}(n)\right| \bigg| \sum_{n \sim M} f_0(n) \overline{\chi(n)}\bigg|\\
&\quad \quad \quad \quad\cdot\left|\mathbb{E}_{n \in [N]_q} f(n) \overline{\chi}(n) - \mathbb{E}_{b \in \mathbb{Z}_q^\times} g(b) \overline{\chi}(b)\right|\bigg).
\end{aligned}
\end{align}

Let us first consider the right-hand side. Let the multiplicative function $h$ be as in Lemma~\ref{le:1astchi_d}. We split the characters into two sets
\[
\mathcal{X}_1 = \left\{\chi \pmod{q} \colon \left|\sum_{n \sim M} f_0(n) \overline{\chi}(n)\right| < (\log q)^{-10} M L(1, \psi) \frac{\varphi(q)}{q}\prod_{\substack{2<p \leq q \\ \psi(p) = 1}} \left(1-\frac{2}{p}\right) \right\}
\]
and
\[
\mathcal{X}_2 = \left\{\chi \pmod{q} \colon \chi \not \in \mathcal{X}_1\right\}.
\]
By the definition of $\mathcal{X}_1$, Proposition~\ref{prop:f=g+h}(iii) and Lemma~\ref{le:MVT}, we see that 
\begin{align}
\label{eq:viiX1}
\begin{aligned}
&\frac{1}{\varphi(q)} \sum_{\chi \in \mathcal{X}_1} \left|\mathbb{E}_{n \in [N]_q} f(n) \overline{\chi}(n)\right| \left| \sum_{n \sim M} f_0(n) \overline{\chi(n)}\right| \left|\mathbb{E}_{n \in [N]_q} f(n) \overline{\chi}(n) - \mathbb{E}_{b \in \mathbb{Z}_q^\times} g(b) \overline{\chi}(b)\right| \\
&\ll (\log q)^{-10} M L(1, \psi) \frac{\varphi(q)}{q}\prod_{\substack{2<p \leq q \\ \psi(p) = 1}} \left(1-\frac{2}{p}\right) \frac{1}{\varphi(q)} \sum_{\chi \pmod{q}} \left|\mathbb{E}_{n \in [N]_q} f(n) \overline{\chi}(n)\right|^2 \\
&\ll   (\log q)^{-10} M L(1, \psi) \frac{\varphi(q)}{q}\prod_{\substack{2<p \leq q \\ \psi(p) = 1}} \left(1-\frac{2}{p}\right) \frac{1}{\varphi(q)} (N+\varphi(q)) \frac{\log N}{N} \\
&\ll \frac{1}{\varphi(q) (\log q)^{9}} M L(1, \psi) \frac{\varphi(q)}{q}\prod_{\substack{2<p \leq q \\ \psi(p) = 1}} \left(1-\frac{2}{p}\right).
\end{aligned}
\end{align}
On the other hand, by the definition of $\mathcal{X}_2$, Lemma~\ref{le:Hal-MonL1chi}, and Lemma~\ref{le:QRupper}, we obtain
\begin{align*}
&\left((\log q)^{-10} M L(1, \psi) \frac{\varphi(q)}{q}\prod_{\substack{2<p \leq q \\ \psi(p) = 1}} \left(1-\frac{2}{p}\right) \right)^{2} |\mathcal{X}_2| \leq  \sum_{\chi \in \mathcal{X}_2} \left|\sum_{n \sim M} f_0(n) \overline{\chi}(n)\right|^{2} \\
&\ll \left(M L(1, \psi) \frac{\varphi(q)}{q}\prod_{\substack{2<p \leq q \\ \psi(p) = 1}} \left(1-\frac{2}{p}\right) + Mq^{-\eta} |\mathcal{X}_2|\right) \sum_{\substack{p \sim M\\(p,q)=1}} (1 \ast \psi)(p) \\
&\ll \left(M L(1, \psi) \frac{\varphi(q)}{q}\prod_{\substack{2<p \leq q \\ \psi(p) = 1}} \left(1-\frac{2}{p}\right) + Mq^{-\eta} |\mathcal{X}_2|\right) M L(1, \psi) \frac{\varphi(q)}{q}\prod_{\substack{2<p \leq q \\ \psi(p) = 1}} \left(1-\frac{2}{p}\right).
\end{align*}
By Siegel's bound $L(1, \psi) \gg q^{-\eta/10}$, so the second term cannot dominate and so $|\mathcal{X}_2| \ll (\log q)^{20}$ and moreover
\begin{equation*}
\sum_{\chi \in \mathcal{X}_2} \left|\sum_{n \sim M} f_0(n) \overline{\chi}(n)\right|^{2} \ll \left(M L(1, \psi) \frac{\varphi(q)}{q}\prod_{\substack{2<p \leq q \\ \psi(p) = 1}} \left(1-\frac{2}{p}\right)\right)^2.
\end{equation*}
Furthermore using the bound $|\mathcal{X}_2| \ll (\log q)^{20}$ and Lemma~\ref{le:Hal-Mon}(i), we obtain similarly to~\eqref{eq:fnsecondHM} that
\[
\sum_{\chi \in \mathcal{X}_2}\left|\mathbb{E}_{n \in [N]_q} f(n) \overline{\chi}(n)\right|^2 \ll 1.
\]
Combining these with Proposition~\ref{prop:f=g+h}(ii) and Cauchy--Schwarz, we obtain 
\begin{align*}
&\frac{1}{\varphi(q)} \sum_{\chi \in \mathcal{X}_2} \left|\mathbb{E}_{n \in [N]_q} f(n) \overline{\chi}(n)\right|  \left|\sum_{n \sim M} f_0(n) \overline{\chi}(n)\right| \left|\mathbb{E}_{n \in [N]_q} f(n) \overline{\chi}(n) - \mathbb{E}_{b \in \mathbb{Z}_q^\times} g(b) \overline{\chi}(b)\right| \\
&\ll \delta \frac{1}{\varphi(q)} \left(\sum_{\chi \in \mathcal{X}_2} \left|\sum_{n \sim M} f_0(n) \overline{\chi}(n)\right|^{2}\right)^{\frac{1}{2}}  \left(\sum_{\chi \in \mathcal{X}_2}\left|\mathbb{E}_{n \in [N]_q} f(n) \overline{\chi}(n)\right|^2\right)^{1/2} \\
&\ll \frac{\delta}{\varphi(q)} M L(1, \psi) \frac{\varphi(q)}{q}\prod_{\substack{2<p \leq q \\ \psi(p) = 1}} \left(1-\frac{2}{p}\right).
\end{align*}
Combining with~\eqref{eq:f0ff-f0ggQR} and~\eqref{eq:viiX1}, we obtain
\[
\frac{1}{|[N]_q|^2} \sum_{n \equiv a \pmod{q}} (f \ast f \ast f_0)(n) = \frac{1}{\varphi(q)^2} (g \ast g \ast f_0)(a) + O\left(\frac{\delta}{\varphi(q)} M L(1, \psi) \frac{\varphi(q)}{q}\prod_{\substack{2<p \leq q \\ \psi(p) = 1}} \left(1-\frac{2}{p}\right)\right).
\]
Recall $A = \{a \in \mathbb{Z}_q^\times \colon |g(a)| \geq \varepsilon/10\}$ in Proposition~\ref{prop:transfer}, so the main term on the right-hand side is 
\[
\geq \frac{1}{\varphi(q)^2}\sum_{\substack{p a_1 a_2 \equiv a \pmod{q} \\ p \sim M, \psi(p) = 1}} g(a_1) g(a_2) \geq \frac{\varepsilon^2}{\varphi(q)^2} \sum_{\substack{p a_1 a_2 \equiv a \pmod{q} \\ p \sim M, \psi(p) = 1 \\ a_1, a_2 \in A}} 1,
\]
and the claim follows.
\end{proof}

Now Proposition~\ref{prop:A'binter-vii} finally follows from Proposition~\ref{prop:transferQR} since~\eqref{eq:cA'B'} implies that
\[
\sum_{\substack{p a_1 a_2 \equiv a \pmod{q} \\ p \sim M, \, \psi(p) = 1 \\ a_1, a_2 \in A}} 1 \geq \sum_{\substack{cp \equiv a \pmod{q} \\ p \sim M, \, \psi(p) = 1 \\ c \in A' \cdot B'}} (\mathbf{1}_A \ast \mathbf{1}_A)(c) \geq \frac{\varepsilon \varphi(q)}{10} \sum_{\substack{c p \equiv a \pmod{q} \\ p \sim M, \, \psi(p) = 1 \\ c \in A' \cdot B'}} 1.
\]

\section{Size of \texorpdfstring{$E_2(q)$}{E2(q)}}
\begin{proof}[Proof of Theorem~\ref{thm:E2>1/2}] Let $\varepsilon>0$ be small. Recall the definitions of $\theta$ and $\theta_0$ from~\eqref{eq:thetaDef} and~\eqref{eq:theta0Def}. Note that $\theta \geq  3/4-\varepsilon$ if $q$ is cube-free and $\theta\geq 2/3-\varepsilon$ if $q$ is not cube-free. Let us apply Proposition~\ref{prop:A'binter}. If Proposition~\ref{prop:A'binter}(a) holds, then the claim follows immediately (adjusting $\varepsilon$). Hence we can assume that Proposition~\ref{prop:A'binter}(b) holds. Let $A'$ and $B'$ be as there. We apply Lemma~\ref{le:structure} for these sets, taking $\alpha = \alpha'= 3/8-2\varepsilon$, and $\beta = 11/16-5\varepsilon$
if $q$ is cube-free and $\alpha = 1/3-2\varepsilon$, $\alpha' = 3/8-10\varepsilon^{1/2}$ and $\beta = 2/3-7\varepsilon$ if $q$ is not cube-free. If Lemma~\ref{le:structure}(a) holds, the claim follows immediately (adjusting $\varepsilon$) from Proposition~\ref{prop:A'binter}(b.ii)

We may then assume that Lemma~\ref{le:structure}(b) holds. In particular, then  the stabilizer $H$ of $A' \cdot B'$ has index $Y=3k+2$ with $k \in \{0, 1, 2\}$ (in case $q$ is cube-free, this follows from Lemma~\ref{le:structure}(b.iii) and in case $q$ is not cube-free first Lemma~\ref{le:structure}(b.i) implies that $Y \leq 12$ so that $k \in \{0, 1, 2, 3\}$ and then~\eqref{eq:alpha'} implies $k \neq 3$). Moreover, by Lemma~\ref{le:structure}(b.ii) there exist cosets $a_1H, \dotsc, a_{k+1}H$ and $b_1H, \dotsc, b_{2k+1}H$ such that $A'$ and $B'$ are contained in $\cup_{j=1}^{k+1} a_j H$ (by Proposition~\ref{prop:A'binter}(b.v) $A'$ and $B'$ are contained in the same $k+1$ cosets) and
\[
\left(\bigcup_{j=1}^{k+1} a_j H\right) \cdot \left(\bigcup_{j=1}^{k+1} a_j H\right) = \bigcup_{j=1}^{2k+1} b_j H = A' \cdot B'.
\]
Furthermore by Proposition~\ref{prop:A'binter}(b.ii) and (b.vi)
\begin{equation}
\label{eq:E2capcupsize}
\left|E_2(q)\cap \left(\bigcup_{j=1}^{2k+1} b_j H\right)\right| \geq \left(\frac{2k+1}{Y}-\varepsilon\right)\varphi(q)
\end{equation}
and
\begin{equation}
\label{eq:pnotinunion}
\sum_{\substack{p \leq q \\ p \not \in \cup_{j=1}^{k+1} a_j H}} 1 \leq (2k+1)\cdot 2\varepsilon \frac{q}{\log q}.
\end{equation}

Let us next observe that there exists $y \in [q^{\varepsilon}, q]$ and $a_0$ such that $a_0 \not \in \cup_{j=1}^{k+1} a_j H$ and
\begin{equation}
\label{eq:P1size}
\sum_{\substack{p \sim y \\ p \in a_0 H}} 1 \gg \frac{y}{q^{\eta'}}
\end{equation}
for any $\eta' > 0$; in case $k=0$ and $a_1 \not \in H$ this follows from Lemma~\ref{le:QR} and otherwise it follows as~\eqref{eq:alclaim}. 

By Kneser's theorem (Lemma~\ref{le:Kneser}),
\[
\left|\left(\bigcup_{j=0}^{k+1} a_j H\right) \cdot \left(\bigcup_{j=1}^{k+1} a_j H\right)\right| \geq (k+2)|H| + (k+1)|H| - |H| = (2k+2)|H| > \left|\bigcup_{j=1}^{2k+1} b_j H\right|.
\]
Hence there exists $j_0 \in \{1, \dotsc, k+1\}$ such that 
\begin{equation}
\label{eq:a0ajonot}
a_0 a_{j_0} H \not \in \bigcup_{j = 1}^{2k+1} b_j H.
\end{equation}

Let
\[
\mathcal{P}_1 = \{p \sim y \colon p \in a_0 H\}
\]
and
\[
\mathcal{P}_2 = \{p \leq q \colon p \in a_{j_0} H\}.
\]
Now by the prime number theorem, Lemma~\ref{le:P2} and~\eqref{eq:pnotinunion},
\begin{align}
\begin{aligned}
\label{eq:P2lowCons}
|\mathcal{P}_2| &= \sum_{p \leq q} 1 - \sum_{\substack{j=1 \\ j \neq j_0}}^{k+1} \sum_{\substack{p \leq q \\ p \in a_j H}}1 - \sum_{\substack{p \leq q \\ p \not \in \cup_{j=1}^{k+1} a_j H}} 1\geq \left(1- \frac{k}{3k+2} \cdot \frac{2}{\theta_0} - 2\varepsilon(2k+1) +o(1)\right) \frac{q}{\log q} \\
&\geq
\begin{cases}
\left(1-20\varepsilon\right) \frac{q}{\log q} &\text{if $k=0$;} \\
\left(\frac{7}{15}-20\varepsilon\right) \frac{q}{\log q} &\text{if $k=1$;} \\
\left(\frac{1}{3}-20\varepsilon\right) \frac{q}{\log q} & \text{if $k = 2$.}
\end{cases}
\end{aligned}
\end{align}
Note that, for any $a \in \mathbb{Z}_q^\times$, we have
\begin{align}\label{eq:energy}
\sum_{\substack{a = p_1 p_2 \\ p_1 \in \mathcal{P}_1\\p_2 \in \mathcal{P}_2}} 1 > 0 \implies a \in E_2(q) \cap a_0a_{j_0} H.
\end{align}
By the Cauchy--Schwarz inequality, 
\begin{align}
\label{eq:multen}
(|\mathcal{P}_1||\mathcal{P}_2|)^2 = \left(\sum_{a \in \mathbb{Z}_q^\times} \sum_{\substack{a = p_1 p_2 \\ p_1 \in \mathcal{P}_1\\p_2\in \mathcal{P}_2}} 1 \right)^2 \leq \left|E_2(q) \cap a_0a_{j_0} H\right| \sum_{a \in \mathbb{Z}_q^\times} \left(\sum_{\substack{a = p_1 p_2 \\ p_1 \in \mathcal{P}_1\\p_2\in \mathcal{P}_2}} 1 \right)^2.
\end{align}
This is a multiplicative variant of~\cite[second inequality in (2.8)]{tao-vu} --- on the right-hand side we have the multiplicative energy $E^{\times}_q(\mathcal{P}_1,\mathcal{P}_2)$, where $E^{\times}_q(B,C)$ is defined between two finite sets $B,C\subseteq \{n \in \mathbb{Z} \colon (n, q) = 1\}$ by
\begin{align*}
E_q^{\times}(B,C)=\sum_{\substack{b_1,b_2\in B, c_1,c_2\in C\\
b_1c_1\equiv b_2c_2\pmod q}} 1.
\end{align*}

Then by~\eqref{eq:multen}
\begin{align}
\label{eq:E2qmult}
\left|E_2(q) \cap a_0a_{j_0} H\right| \geq \frac{(|\mathcal{P}_1||\mathcal{P}_2|)^2}{E_q^{\times}(\mathcal{P}_1,\mathcal{P}_2)}.    
\end{align}
Now we show an upper bound for $E_q^{\times}(\mathcal{P}_1,\mathcal{P}_2)$. Let $\lambda_d^+$ be as in Lemma~\ref{le:linear} with level $D = q^\theta$ and $s=1$ (so that $z = D$).

Then by Lemma~\ref{le:linear}(ii)
\begin{align*}
E_q^{\times}(\mathcal{P}_1,\mathcal{P}_2)\leq \sum_{p_1,q_1\in \mathcal{P}_1, p_2 \in \mathcal{P}_2,n_2\leq q}\,\,\sum_{d \mid n_2} \lambda_d^+ 1_{p_1p_2\equiv q_1 n_2 \pmod q}.   
\end{align*}
Thus by the orthogonality of characters
\begin{align}
\label{eq:chibound}
E_q^{\times}(\mathcal{P}_1,\mathcal{P}_2) \leq \frac{1}{\varphi(q)}\sum_{\chi\pmod q}\left|\sum_{p\in \mathcal{P}_1} \chi(p)\right|^2 \sum_{p\in \mathcal{P}_2}\chi(p) \overline{\sum_{n\leq q}\chi(n) \sum_{d \mid n} \lambda_d^+}.    
\end{align}

We first consider the contribution of $\chi=\chi_0$ to~\eqref{eq:chibound}. This is by~\eqref{eq:Simplegcd}, Lemma~\ref{le:linear}(iii) and~\eqref{eq:P2lowCons}
\begin{align}
\begin{aligned}
\label{eq:princE2}
&\leq \frac{1}{\varphi(q)}|\mathcal{P}_1|^2 |\mathcal{P}_2| \left|\sum_{\substack{d \leq D \\ (d, q) = 1}} \lambda_d^+ \left(\frac{\varphi(q)}{d} + O(\tau(q))\right) \right|\leq \frac{1}{\varphi(q)}(2+o(1))|\mathcal{P}_1|^2 |\mathcal{P}_2| \frac{q}{\log z}\\ 
&\leq \frac{2+o(1)}{\theta}|\mathcal{P}_1|^2 |\mathcal{P}_2| \frac{q}{\varphi(q) \log q} \\
&\leq
\begin{cases}
\frac{1}{\varphi(q)}\left(3+100\varepsilon\right) |\mathcal{P}_1|^2|\mathcal{P}_2|^2 & \text{when $k = 0$;} \\
\frac{1}{\varphi(q)}\left(\frac{45}{7}+100\varepsilon\right) |\mathcal{P}_1|^2|\mathcal{P}_2|^2 & \text{when $k = 1$;} \\
\frac{1}{\varphi(q)}\left(9+100\varepsilon\right) |\mathcal{P}_1|^2|\mathcal{P}_2|^2 & \text{when $k = 2$.}
\end{cases}
\end{aligned}
\end{align}

We shall show that
\begin{align}
\label{eq:E2error}
\frac{1}{\varphi(q)} \sum_{\substack{\chi\pmod q \\ \chi \neq \chi_0}}\left|\sum_{p\in \mathcal{P}_1} \chi(p)\right|^2 \left|\sum_{p\in \mathcal{P}_2}\chi(p)\right| \left|\sum_{n\leq q}\chi(n) \sum_{d \mid n} \lambda_d^+\right| \ll \frac{|\mathcal{P}_1|^2|\mathcal{P}_2|^2}{(\log q)^{20}\varphi(q)}.
\end{align}
Before proving~\eqref{eq:E2error}, let us show how it implies Theorem~\ref{thm:E2>1/2}(i)--(ii). Combining~\eqref{eq:E2error} with~\eqref{eq:chibound} and the evaluation of the contribution of the principal character in~\eqref{eq:princE2}, we obtain
\[
E_q^{\times}(\mathcal{P}_1,\mathcal{P}_2) \leq \begin{cases}
\frac{1}{\varphi(q)}\left(3+101\varepsilon\right) |\mathcal{P}_1|^2|\mathcal{P}_2|^2 & \text{when $k = 0$;} \\
\frac{1}{\varphi(q)}\left(\frac{45}{7}+101\varepsilon\right) |\mathcal{P}_1|^2|\mathcal{P}_2|^2 & \text{when $k = 1$;} \\
\frac{1}{\varphi(q)}\left(9+101\varepsilon\right) |\mathcal{P}_1|^2|\mathcal{P}_2|^2 & \text{when $k = 2$.}
\end{cases}
\]
This together with~\eqref{eq:E2qmult},~\eqref{eq:a0ajonot}, and~\eqref{eq:E2capcupsize} implies
\[
|E_2(q)| \geq 
\begin{cases}
\left(\frac{1}{2}-\varepsilon\right) \varphi(q) + \left(\frac{1}{3}-200\varepsilon\right) \varphi(q)  & \text{if $k = 0$;} \\
\left(\frac{3}{5}-\varepsilon\right) \varphi(q) + \left(\frac{7}{45}-200\varepsilon\right) \varphi(q) & \text{if $k = 1$;} \\
\left(\frac{5}{8}-\varepsilon\right) \varphi(q) + \left(\frac{1}{9}-200\varepsilon\right) \varphi(q)  & \text{if $k = 2$}.
\end{cases}
\]
In each case the claim follows on taking $\varepsilon>0$ small enough (since $1/2+1/3,3/5+7/45,5/8+1/9>11/16$).

Let us now turn to proving the remaining claim~\eqref{eq:E2error}. We argue somewhat similarly as in the proof of Proposition~\ref{prop:transfer}(iv). We write
\[
\mathcal{X}_1 = \left\{\chi \pmod{q} \colon \left|\sum_{p \in \mathcal{P}_1} \chi(p)\right| \leq \frac{|\mathcal{P}_1|}{(\log q)^{20}} \right\}
\]
and
\[
\mathcal{X}_2 = \left\{\chi \pmod{q} \colon \chi \neq \chi_0 \text{ and } \left|\sum_{p \in \mathcal{P}_1} \chi(p)\right| > \frac{|\mathcal{P}_1|}{(\log q)^{20}} \right\}.
\]
By the definition of $\mathcal{X}_1$, the Cauchy--Schwarz inequality, the mean value theorem (Lemma~\ref{le:MVT}) and~\eqref{eq:P2lowCons}, we have
\begin{align*}
&\frac{1}{\varphi(q)} \sum_{\substack{\chi \in \mathcal{X}_1}}\left|\sum_{p\in \mathcal{P}_1} \chi(p)\right|^2 \left|\sum_{p\in \mathcal{P}_2}\chi(p)\right| \left|\sum_{n\leq q}\chi(n) \sum_{d \mid n} \lambda_d^+\right| \\
&\ll \frac{1}{\varphi(q)} \frac{|\mathcal{P}_1|^2}{(\log q)^{40}}  \left(\sum_{\substack{\chi\pmod q}} \left|\sum_{p\in \mathcal{P}_2}\chi(p)\right|^2\right)^{1/2} \left(\sum_{\chi \pmod{q}} \left|\sum_{n\leq q}\chi(n) \sum_{d \mid n} \lambda_d^+\right|^2\right)^{1/2} \\
&\ll \frac{1}{\varphi(q)} \frac{|\mathcal{P}_1|^2}{(\log q)^{40}}  (q |\mathcal{P}_2|)^{1/2} (q \cdot q \log q)^{1/2} \ll \frac{|\mathcal{P}_1|^2 |\mathcal{P}_2|^2}{(\log q)^{20}\varphi(q)}.
\end{align*}
For $\chi \in \mathcal{X}_2$, we note that by the Burgess bound (Lemma~\ref{le:Burgess}), for some $\eta > 0$,
\[
\left|\sum_{n\leq q}\chi(n) \sum_{d \mid n} \lambda_d^+\right| = \left|\sum_{d \leq q^\theta} \lambda_d^+ \chi(d) \sum_{n\leq q/d}\chi(n) \right| \ll q^{1-\eta}.
\]
Hence, choosing $L = \lceil \frac{\log q}{\log y}\rceil$ and applying H\"older's inequality, we obtain
\begin{align*}
&\frac{1}{\varphi(q)} \sum_{\substack{\chi \in \mathcal{X}_2}}\left|\sum_{p\in \mathcal{P}_1} \chi(p)\right|^2 \left|\sum_{p\in \mathcal{P}_2}\chi(p)\right| \left|\sum_{n\leq q}\chi(n) \sum_{d \mid n} \lambda_d^+\right| \\
&\ll   \frac{(q^{1-\eta})^{2/L}}{\varphi(q)} \left(\sum_{\substack{\chi \in \mathcal{X}_2}}\left|\sum_{p\in \mathcal{P}_1} \chi(p)\right|^{2L}\right)^{\frac{1}{L}} \left(\sum_{\substack{\chi \in \mathcal{X}_2}}\left|\sum_{p\in \mathcal{P}_2} \chi(p)\right|^{2}\right)^{1/2}  \left(\sum_{\substack{\chi \in \mathcal{X}_2}}\left|\sum_{n\leq q}\chi(n) \sum_{d \mid n} \lambda_d^+\right|^2\right)^{\frac{1}{2}-\frac{1}{L}}.
\end{align*}
Applying Lemma~\ref{le:MVT} to each mean square and using~\eqref{eq:P1size} with $\eta' = \eta/(10L)$, we see that the above expression is
\[
\ll \frac{(q^{1-\eta})^{2/L}}{\varphi(q)} \left(y^{2L} (\log q)^{L-1}\right)^{\frac{1}{L}} (q |\mathcal{P}_2|)^{1/2} \left(q \cdot q \log q \right)^{\frac{1}{2}-\frac{1}{L}} \ll \frac{|\mathcal{P}_1|^2 |\mathcal{P}_2|^2}{q^{\eta/L} \varphi(q)}.
\]
This completes the proof of~\eqref{eq:E2error}.
\end{proof}

\section*{Acknowledgments}
KM was supported by Academy of Finland grant no. 285894. JT was supported by Academy of Finland grant no. 340098, a von Neumann Fellowship (NSF grant no. DMS-1926686), and funding from
the European Union's Horizon
Europe research and innovation programme under Marie Sk\l{}odowska-Curie grant agreement no.
101058904. The authors would like to thank Igor Shparlinski for useful discussions and for pointing out the paper~\cite{Zhao}. The authors are also grateful to the anonymous referee for helpful comments and corrections.

\bibliography{refs}
\bibliographystyle{plain}
\end{document}